\documentclass{amsart}
\usepackage{amsfonts,amsmath,amsthm,amssymb}
\usepackage{mathrsfs}
\usepackage[dvipsnames]{xcolor}
\definecolor{MediumGray}{gray}{0.6}
\usepackage{tikz}
\usepackage{ytableau}
\usepackage[colorlinks=true,linkcolor=blue,citecolor=blue]{hyperref}
\usepackage[curve]{xypic}

\usepackage{multirow}
\usepackage{comment}

\newcommand\TikZ[1]{\begin{matrix}\begin{tikzpicture}#1\end{tikzpicture}\end{matrix}}

\usepackage[bbgreekl]{mathbbol}

\numberwithin{equation}{section}
\theoremstyle{plain}
\newtheorem{theorem}[equation]{Theorem}
\newtheorem{prop}[equation]{Proposition}
\newtheorem{lemma}[equation]{Lemma}
\newtheorem{corollary}[equation]{Corollary}
\theoremstyle{definition}
\newtheorem{example}[equation]{Example}
\newtheorem{remark}[equation]{Remark}
\newtheorem{definition}[equation]{Definition}

\newcommand{\bs}{\mathbf{s}}

\newcommand{\blambda}{\boldsymbol{\lambda}}
\newcommand{\bmu}{\boldsymbol{\mu}}

\newcommand{\bsigma}{\boldsymbol{\sigma}}

\def\bbZ{{\mathbb{Z}}}
\def\bbN{{\mathbb{N}}}

\def\bbk{{\mathbb{k}}}

\def\down{\vee}
\def\up{\wedge}

\def\updown{\mathsf{w}_{\up\down}}
\def\cupdown{\mathsf{c}_{\up\down}}

\begin{document}
\title[Decomposition numbers for unipotent blocks with small $\mathfrak{sl}_2$-weight]{Decomposition numbers for unipotent blocks with small  $\mathfrak{sl}_2$-weight in finite classical groups}
\author{Olivier Dudas}
\address[O.D.]{CNRS \& Aix-Marseille Universit\'e \\
Institut de Math\'ematiques de Marseille \\
Campus de Luminy, 13009 Marseille, France}
\email{olivier.dudas@univ-amu.fr}
\author{Emily Norton}
\address[E.N.]{School of Mathematics, Statistics and Actuarial Science \\ University of Kent, Canterbury, CT2 7FS, United Kingdom}
\email{E.Norton@kent.ac.uk}

\begin{abstract}
We show that parabolic Kazhdan-Lusztig polynomials of type $A$ compute the decomposition numbers in certain Harish-Chandra series of unipotent characters of finite groups of Lie types $B$, $C$ and $D$ over a field of non-defining characteristic $\ell$. Here, $\ell$ is a ``unitary prime" -- the case that remains open in general. The bipartitions labeling the characters in these series are small with respect to $d$, the order of $q$ mod $\ell$, although they occur in blocks of arbitrarily high defect. Our main technical tool is the categorical action of an affine Lie algebra on the category of unipotent representations, which identifies the branching graph for Harish-Chandra induction with the $\widehat{\mathfrak{sl}}_d$-crystal on a sum of level $2$ Fock spaces. Further key combinatorics has been adapted from Brundan and Stroppel's work on Khovanov arc algebras to obtain the closed formula for the decomposition numbers in a $d$-small Harish-Chandra series. 
\end{abstract}

\maketitle

\section*{Introduction}

For a finite group $G$, decomposition numbers encode how ordinary irreducible representations behave after reduction  modulo a prime number $\ell$. In this paper, we study the case where $G$ is one the finite classical groups $\mathrm{SO}_{2n+1}(q)$,  $\mathrm{Sp}_{2n}(q)$,  
$\mathrm{O}_{2n}^{\pm}(q)$ and $\ell$ is prime to $q$ (non-defining characteristic). 

\smallskip

Computing the whole decomposition matrix can be done blockwise, and using the Jordan decomposition one can often restrict to the unipotent blocks, which are the ones containing unipotent characters (see \cite[Thm. 11.8]{BR02}). However, even for those blocks, computing  the decomposition numbers explicitly seems out of reach. Still, for finite general linear groups, they can be related to Kazhdan--Lusztig polynomials, at least when $\ell$ is large enough with respect to $n$. Our main result gives a similar interpretation for classical groups, see Corollary~\ref{cor:final}, but in a much more restrictive case. 

\smallskip

Let $d$ be the multiplicative order of $q$ in $\mathbb{F}_\ell^\times$. When $d$ is odd (the \emph{linear prime} case), the representation theory of unipotent blocks of classical groups is governed by that of finite general linear groups. When $d$ is even (the \emph{unitary prime} case), much less is known. For example, the decomposition numbers for unipotent blocks of $\mathrm{Sp}_{2n}(q)$ in the unitary case have been determined up to $n = 4$ only, and quite recently: in 1998 for $n=2$ \cite{OW98}, in 2014 for $n=3$ \cite{HN14} and in 2022 for $n=4$ \cite{DM23}. The standard strategy for computing these numbers is the following: 
\begin{itemize}
 \item[(1)] Compute the projective cover of cuspidal  representations ;
  \item[(2)] Decompose the module obtained by Harish-Chandra induction of that projective cover.
\end{itemize}

\smallskip

We will work under two assumptions which will allow to solve these two problems: first, we will restrict our ordinary irreducible characters to a given Harish-Chandra series.  Second, we will impose a smallness condition with respect to $d$ on the Lusztig symbols in the series.  The exact condition is given in Section~\ref{ssec:dsmall}.  In this setting, cuspidal representations, when they occur, are always the ``smallest'' within the series with respect to the partial order on symbols. Thus problem (1) is no problem thanks to the unitriangular shape of the decomposition matrix. Everything then hinges on problem (2) -- and it turns out to be tractable for these series. Note that we end up working with blocks of any defect, even though in our situation the decomposition numbers turn out to be either 0 or 1.  

\smallskip
Let us be more precise on the behaviour of (2). Using the $\widehat{\mathfrak{sl}}_d$-action constructed in \cite{DVV17}, one can break the Harish-Chandra induction functor into pieces, and obtain an $i$-induction functor $F_i$ for each $i \in \mathbb{Z}/d\mathbb{Z}$. If $P_S$ is the projective cover of a simple module $S$ of highest weight (for the action of $\mathfrak{sl}_2$ corresponding to some $i  \in \mathbb{Z}/d\mathbb{Z}$), then $F_i^n(P_S)$ is very close to being the projective cover of $F_i^n(S)$. The difference involves only characters which are in the image of $F_i^{n+1}$, see Proposition~\ref{prop:decnumbers} for an exact statement. With our assumption on the Harish-Chandra series, all these extra characters lie in another series, hence do not contribute to the decomposition numbers we are interested in.

\smallskip
The explicit computation of the decomposition numbers now boils down to knowing exactly how to compute the $i$-induction on unipotent characters and simple modules in characteristic~$\ell$. Under our assumptions, one can attach to each irreducible character/module an up-down diagram and a cup diagram as in \cite{BS1,Bru} and obtain an elementary description of the action of $F_i$ on these combinatorial data as in \cite{BS3}. See Section \ref{sec:diagrammatics}, which takes its cue from the analogous situation for Hecke algebras at $d=\infty$ \cite{BS3,Bru}. This combinatorial set-up positions us to prove our main theorem, Theorem~\ref{thm:main}, in an inductive way using Proposition \ref{prop:decnumbers}. The latter exploits the small highest weight of the simple $\mathfrak{sl}_2$-modules in which such
 characters are found with respect to the categorical action on the category of representations. We observe that we get the exact same formula as in \cite{Bru,BS1}, thus relating the decomposition numbers within the Harish-Chandra series to the parabolic Kazhdan--Lusztig polynomials for a maximal parabolic in type $A$.

\section{Categorical $\mathfrak{sl}_2$-action}

In this section we prove a formula relating the character of the projective cover of a simple module $S$ and the character of the projective cover of its highest weight support. Only part of the character can be controlled, but this will be sufficient to show in Section~\ref{sec:proof} how to compute decomposition numbers from the case of cuspidal simple modules.

\subsection{Recollection on categorical actions}
We recall here some of the features of $\mathfrak{sl}_2$-actions on categories as defined in \cite{CR08,R12}.

\smallskip

Let $\mathbb{\Lambda}$ be a ring with unit and $\mathcal{V}$ be a $\mathbb{\Lambda}$-linear abelian category.
A \emph{categorical datum} on $\mathcal{V}$ is given by a pair of biadjoint exact endofunctors $E$ and $F$
of $\mathcal{V}$, together with two natural transformations $X \in \mathrm{End}(E)$ and $T \in \mathrm{End}(E^2)$
satisfying the relations given in \cite[3.3.3]{R12} in the case of $\mathfrak{sl}_2$. Equivalently, we require that $X$ and $T$ induce an action of the affine nil-Hecke algebra of $\mathfrak{S}_n$
on the functor $E^n$ for all $n \geq 0$. Using that structure one can define the divided power functors $E^{(n)}$ and  $F^{(n)}$ 
which are still exact and biadjoint. They satisfy
$$ E^n \simeq \big(E^{(n)}\big)^{\oplus n!} \quad \text{and} \quad F^n \simeq \big(F^{(n)}\big)^{\oplus n!}.$$

Assume now that $\mathbb{\Lambda}$ is a field and that  $\mathcal{V}$ has finite length. 
An \emph{$\mathfrak{sl}_2$-categorical action} on $\mathcal{V}$ is given by a categorical datum
$(E,F,X,T)$ and a decomposition 
$$\mathcal{V} = \bigoplus_{\omega \in \mathbb{Z}} \mathcal{V}_\omega$$
of $\mathcal{V}$ into abelian categories (which we will call \emph{weight categories}).
Furthermore, the functors $E$ and $F$ should shift the weights by $2$ and $-2$ respectively
 $$ \xymatrix{ \mathcal{V}_\omega  \ar@/^/[rr]^{E}
 \ar@/_/@{<-}[rr]_{F} & & \mathcal{V}_{\omega + 2}}$$
such that in the Grothendieck group $K_0(\mathcal{V}_\omega)$, the commutator $[E][F] - [F][E]$ acts by multiplication by $\omega$. 
In particular the class $e = [E]$ and $f=[F]$ of the functors in the complexified Grothendieck group
$V = \mathbb{C}\otimes_\mathbb{Z} K_0(\mathcal{V})$ induce an action of $\mathfrak{sl}_2$ for which
the weight space of weight $\omega$ is exactly $V_\omega =  \mathbb{C}\otimes_\mathbb{Z} K_0(\mathcal{V}_\omega)$.
Note that we will always assume that this action is integrable, so that $e$ and $f$ are locally nilpotent. 
 
\smallskip

For such a notion of $\mathfrak{sl}_2$-action on a category, the divided power functors satisfy the following identity on weight spaces of weight $\omega \geq 0$, see \cite[Lem. 4.8]{R12}
\begin{equation}
\label{eq:Mackey}
{E^{(n)} F^{(n)}}_{\big|\mathcal{V}_\omega}  \simeq 1_{\mathcal{V}_\omega}^{\oplus \binom{\omega}{n}} \oplus \big(FE_{\big|\mathcal{V}_\omega}\big)^{\oplus \binom{\omega}{n-1}}  \oplus \big({F^{(2)}E^{(2)}}_{\big|\mathcal{V}_\omega}\big)^{\oplus \binom{\omega}{n-2}} \oplus \cdots \oplus  {F^{(n)}E^{(n)}}_{\big|\mathcal{V}_\omega}.
\end{equation}

We now assume that $\mathrm{End}(S) \simeq \mathbb{\Lambda}$ for all simple objects $S \in \mathrm{Irr}\,\mathcal{V}$. In that case 
it is proven in \cite[Prop. 5.20]{CR08} that $E(S)$, if non-zero, has simple socle and head and that they are isomorphic. Successive applications
of $E$ give a highest weight semi-simple object. We will use the further following properties, which are also proved in \cite[Prop. 5.20]{CR08}.

\begin{lemma}
\label{lem:cat}
Let $S \in \mathrm{Irr}\,\mathcal{V_\omega}$ and $n \geq 0$ be such that $E^{n+1}(S) = 0$ and $E^n(S) \neq 0$.
\begin{enumerate} 
 \item $E^{(n)} (S)$ is simple.
 \item The socle and head of $F^{(n)} E^{(n)}(S)$ are isomorphic to $S$. 
 \item The simple module $S$ occurs in $F^{(n)} E^{(n)}(S)$ with multiplicity $\binom{\omega+2n}{n}$ as a composition factor.
\end{enumerate}
\end{lemma}

\subsection{Decomposition numbers}\label{subsec:decnumbers}
Let $\mathbb{O}$  be a complete discrete valuation ring, with residue field $\mathbb{k}$ of positive characteristic and fraction field $\mathbb{K}$ of characteristic zero.
Let $\{G_r\}_{r \in \mathbb{N}}$ be a family of finite groups. We consider the category 
$$\mathbb{\Lambda} \mathcal{G} =  \bigoplus_{r \geq 0} \mathbb{\Lambda}G_r \mathsf{-mod}$$
which is the sum of the categories of finitely generated representations of $G_r$ over $\mathbb{\Lambda}$, where $\mathbb{\Lambda}$ is any ring among $\mathbb{K},\mathbb{O},\mathbb{k}$.  If $\mathbb{k}$ and $\mathbb{K}$ are large enough for all the finite groups encountered, the following conditions will be satisfied:
\begin{itemize}
\item For $\mathbb{\Lambda} = \mathbb{K}, \mathbb{k}$, the category $\mathbb{\Lambda} \mathcal{G}$ has finite length and $\mathrm{End}(S) = \mathbb{\Lambda}$ for all $S \in \mathrm{Irr}_\mathbb{k} \mathcal{G}$.
\item Every $S \in \mathrm{Irr}_\mathbb{k} \mathcal{G}$ has a projective cover $P_S$ in $\mathbb{k} \mathcal{G}$, unique up to isomorphism. 
\item Every projective module $P$ in $\mathbb{k} \mathcal{G}$ lifts uniquely to a projective module $\widetilde P$ in $\mathbb{O} \mathcal{G}$. 
\item $\mathbb{K} \mathcal{G}$ is semisimple.
\end{itemize}
If $S \in \mathrm{Irr}_\mathbb{k} \mathcal{G}$ and $\Delta\in \mathrm{Irr}_\mathbb{K} \mathcal{G}$, the decomposition number of $S$ in $\Delta$ is the multiplicity of $\Delta$ as a direct summand (equivalently, a composition factor) of $\mathbb{K} \otimes_\mathbb{O} \widetilde{P_S}$. We denote it by
$$[P_S : \Delta]$$
in a way which will look familiar to the reader interested in highest weight categories.

\smallskip

Now, let $\mathcal{V}$ be a direct summand of $\mathbb{O}\mathcal{G}$. We assume that $(E,F,X,T)$ is a categorical datum on $\mathcal{V}$ inducing an $\mathfrak{sl}_2$-categorical action on $\mathbb{k} \mathcal{V}$. This is to ensure that the divided powers $E^{(n)}$ and $F^{(n)}$ defined in the previous section can be lifted to exact and biadjoint endofunctors of $\mathcal{V}$, even though $\mathcal{V}$ itself does not have an action of $\mathfrak{sl}_2$ since its Grothendieck group might be too big. 
The image of these functors by extension of scalars will be still denoted by $E^{(n)}$ and $F^{(n)}$.

\begin{prop}\label{prop:decnumbers}
In the previous setting, 
 let $S \in \mathrm{Irr}_\mathbb{k} \mathcal{V}$ and $n \geq 0$ be such that $E^{n+1}(S) = 0$ and $E^n(S) \neq 0$. Then 
 $$ [P_S : \Delta] = [P_{E^{(n)}S} : E^{(n)} \Delta]$$
 for all irreducible characters $\Delta \in \mathrm{Irr}_\mathbb{K} \mathcal{V}$ such that $E^{n+1} \Delta = 0$.
\end{prop}

\begin{proof} Let $\omega$ be the weight of the simple module $T = E^{(n)}S$. Note that $\omega\geq 0$ since the class of $T$ is a highest weight vector.
Since $S$ is the head of $F^{(n)} E^{(n)} S$   by Lemma \ref{lem:cat}(2)
we have that $F^{(n)} P_{T}$ contains $P_S$ as 
a direct summand. We write
\begin{equation}
 \label{eq:Q}
 F^{(n)} P_{E^{(n)}S} = P_S \oplus Q
\end{equation}
for some projective module $Q$. We want to have some control on the character of $Q$, so for that we compute
the image of \eqref{eq:Q} by $E^{(n)}$. Using \eqref{eq:Mackey} we have
$$ E^{(n)} P_S \oplus  E^{(n)} Q = E^{(n)} F^{(n)} P_{T} = P_{T}^{\oplus \binom{\omega}{n}} \oplus FE P_{T}^{\oplus \binom{\omega}{n-1}}\oplus\cdots \oplus F^{(n)} E^{(n)} P_{T}. $$
Now we claim that $E^{(n)} P_S$ contains $\binom{\omega}{n}$ copies of $P_T$. Indeed, the weight of $S$ equals 
$\omega - 2n$ and we have
$$\mathrm{Hom}(E^{(n)} P_S,T) = \mathrm{Hom}(E^{(n)} P_S,E^{(n)} S) \simeq  \mathrm{Hom}(P_S,F^{(n)} E^{(n)} S)$$
which has dimension $\binom{\omega}{n}$ by Lemma~\ref{lem:cat}.(3). This proves that $E^{(n)} Q$ is a direct summand of a sum of 
modules of the form $F^{(k)} E^{(k)}(P_T)$ where $k \geq 1$. By the lifting property of projective modules, the same holds over $\mathbb{O}$.

\smallskip

Now, let $\Delta$ be an irreducible character of $\mathbb{K}\mathcal{V}$ such that $[Q : \Delta] \neq 0$.
In other words, $\Delta$ is isomorphic to a submodule of $\mathbb{K} \otimes_\mathcal{O} \widetilde{Q}$, and since $E^{(n)}$
is exact, $E^{(n)} \Delta $ is isomorphic to  a submodule of $\mathbb{K} \otimes_\mathcal{O} E^{(n)}\widetilde{Q}$. Furthermore, it 
must be non-zero since by definition $\widetilde{Q}$ is a direct summand of $F^{(n)} \widetilde{P}_T$, putting $\Delta$ in the image of $F^{(n)}$. Now by the previous paragraph, there is $k \geq 1$ such that 
$[F^{(k)} E^{(k)}(P_T) : E^{(n)} \Delta] \neq 0$, which forces $E^{(k)} E^{(n)} \Delta$, and therefore $E^{n+k} \Delta$ to be non-zero.
We showed that $E^{n+1} \Delta = 0$ implies  $[Q : \Delta] = 0$, which gives the result by \eqref{eq:Q}. 
\end{proof}

\begin{remark} The same result (with analogous proof) would hold for highest weight categories and standard objects $\Delta$.
\end{remark}

\section{Unipotent representations of finite classical groups}

By a \emph{finite classical group} $G_n(q)$ we will always mean one of the following groups:
$$\begin{array}{l|cccc}
 \text{Group} & \mathrm{SO}_{2n+1}(q) &  \mathrm{Sp}_{2n}(q) & \mathrm{O}_{2n}^+(q) & \mathrm{O}_{2n}^-(q)\\[3pt]\hline
 \text{Type} \phantom{\Big)}& B_n & C_n & D_n & {}^2 D_n \end{array}$$
where $q > 1 $ is a power of some odd prime. The convention for small values of $n$ is that $ \mathrm{SO}_{1}(q) =\mathrm{Sp}_{0}(q) = 
 \mathrm{O}_{0}^{\pm}(q) = 1$ is the trivial group, $\mathrm{O}_2^{+}(q) = \mathrm{GL}_1(q) \rtimes \mathbb{Z}/2$ and $\mathrm{O}_2^{-}(q) = \mathrm{GU}_1(q) \rtimes \mathbb{Z}/2$ are $1$-dimensional tori (of respective orders $q-1$ and $q+1$) extended by $\mathbb{Z}/2$.
\smallskip

We will work with modular representations, therefore with different coefficient rings for our representations.
As in Section~\ref{subsec:decnumbers} we fix a complete discrete valuation ring $\mathbb{O}$ with residue field $\mathbb{k}$ and fraction field $\mathbb{K}$. We assume that $\mathbb{K}$ has characteristic zero, $\mathbb{k}$ has characteristic $\ell > 0$ and that they are both large enough for all the finite groups encountered. In particular, every irreducible representation over $\mathbb{K}$ or $\mathbb{k}$ will be absolutely irreducible.
We will always work under the assumption that $\ell$ is odd and does not divide $q$. The multiplicative order $d$ of $q$ in $\mathbb{k}^\times$ will be assumed to be even, so that we work in the \emph{unitary prime} case.

\smallskip

Note that some of the results contained in this section for the groups $\mathrm{O}_{2n}^{\pm}(q)$
have not been published yet. Still, we have decided to include these groups 
since the combinatorics is not much different from the groups of type $B/C$, 
and since there are work in progress by Li--Shan--Zhang \cite{LSZ23} and Cia-Luvecce \cite{CL} which will contain all the results we need
here.

\subsection{Combinatorics}\label{subsec:comb}
We recall here the combinatorics that is used to classify the unipotent characters and unipotent blocks of finite classical groups in the unitary prime case.

\subsubsection{Partitions and symbols}
Let $m$ be a non-negative integer. A partition $\lambda$ of $m$ is a non-increasing sequence of non-negative integers $\lambda = (\lambda_1 \geq \lambda_2
\geq \cdots \geq 0)$ which add up to $m$, the size $|\lambda|$ of $\lambda$. A bipartition $\blambda = \lambda^1.\lambda^2$ of $m$ is a pair of partitions
$(\lambda_1,\lambda_2)$ such that $|\lambda^1| + |\lambda^2| = m$. A charged bipartition $|\blambda,\bs\rangle$ is the data of a bipartition and a pair of integers $\bs = (s_1,s_2) \in
\mathbb{Z}^2$, called the charge. 

\smallskip

To a charged bipartition $|\blambda,\bs\rangle$ one can attach a charged symbol $\Theta(\blambda,\bs)$ corresponding to the pair of charged $\beta$-sets coming from the two partitions. More precisely, we have $\Theta(\blambda,\bs) = (X_1,X_2)$ where 
$$X_k = \{ s_k + \lambda_j -j+1 \, |\, j\geq 1\}.$$
A symbol will be represented by the corresponding $2$-abacus, with the first row $X_1$ on the bottom. The \emph{defect} of a charged symbol $\Theta= \Theta(\blambda,\bs)$ is $\mathsf{def}(\Theta) = s_2-s_1$.

\begin{example}\label{ex:symbol}
The charged symbol of charge $(-4,3)$ and bipartition $1^3.2^21$ will be represented as follows
\[
\TikZ{[scale=.5]
\draw
(-7,2)node[]{\hbox{row 2}}
(-7,1)node[]{\hbox{row 1}}
(9,2)node[fill,circle,inner sep=.5pt]{}
(9,1)node[fill,circle,inner sep=.5pt]{}
(8,2)node[fill,circle,inner sep=2pt]{}
(8,1)node[fill,circle,inner sep=.5pt]{}
(7,2)node[fill,circle,inner sep=2pt]{}
(7,1)node[fill,circle,inner sep=.5pt]{}
(6,2)node[fill,circle,inner sep=.5pt]{}
(6,1)node[fill,circle,inner sep=.5pt]{}
(5,2)node[fill,circle,inner sep=2pt]{}
(5,1)node[fill,circle,inner sep=.5pt]{}
(4,2)node[fill,circle,inner sep=.5pt]{}
(4,1)node[fill,circle,inner sep=.5pt]{}
(3,2)node[fill,circle,inner sep=2pt]{}
(3,1)node[fill,circle,inner sep=.5pt]{}
(2,2)node[fill,circle,inner sep=2pt]{}
(2,1)node[fill,circle,inner sep=.5pt]{}
(1,2)node[fill,circle,inner sep=2pt]{}
(1,1)node[fill,circle,inner sep=.5pt]{}
(0,2)node[fill,circle,inner sep=2pt]{}
(0,1)node[fill,circle,inner sep=2pt]{}
(-1,2)node[fill,circle,inner sep=2pt]{}
(-1,1)node[fill,circle,inner sep=2pt]{}
(-2,2)node[fill,circle,inner sep=2pt]{}
(-2,1)node[fill,circle,inner sep=2pt]{}
(-3,2)node[fill,circle,inner sep=2pt]{}
(-3,1)node[fill,circle,inner sep=.5pt]{}
(-4,2)node[fill,circle,inner sep=2pt]{}
(-4,1)node[fill,circle,inner sep=2pt]{}
(9,0)node[]{6}
(8,0)node[]{5} 
(7,0)node[]{4}
(6,0)node[]{3}
(5,0)node[]{2}
(4,0)node[]{1}
(3,0)node[]{0}
(2,0)node[]{-1} 
(1,0)node[]{-2}
(0,0)node[]{-3}
(-1,0)node[]{-4}
(-2,0)node[]{-5}
(-3,0)node[]{-6} 
(-4,0)node[]{-7}
;
}
\]
\end{example}

\subsubsection{Adding and removing boxes}
The unipotent characters will be parametrized by such symbols. In order to explain how the induction and restriction of unipotent representations behave on symbols, we will need the notion of addable/removable boxes. Recall that $d$ is a fixed even integer. Let $i \in \mathbb{Z}/d$ and $\Theta = (X_1,X_2)$ be a charged symbol. An addable $i$-box of $\Theta$ in the row $k \in \{1,2\}$ is an integer $x$ such that
\begin{itemize}
 \item $x \equiv i+ (k-1) \frac{d}{2}$ mod $d$;
 \item $x\in X_k$ and $x+1 \notin X_k$.
\end{itemize}
Adding the $i$-box $x$ in the symbol $\Theta$ consists in replacing $x$ by $x+1$ in $X_k$. 
\smallskip
A removable $i$-box of $\Theta$  in the row $k \in \{1,2\}$ is an integer $x$ such that 
\begin{itemize}
 \item $x \equiv i+ (k-1) \frac{d}{2}$ mod $d$;
 \item $x \notin X_k$ and $x+1 \in X_k$.
\end{itemize}
Removing the $i$-box $x$ in the symbol $\Theta$ consists in replacing $x+1$ by $x$ in $X_k$.  

\smallskip

There is a notion of \emph{good} addable/removable $i$-box, see for example \cite[Section 3]{JMMO}, \cite[Theorem 2.8]{FLOTW}. Given $i$, there is at most one \emph{good} addable/removable $i$-box in a charged symbol, and it can be used to describe the Kashiwara operators $\widetilde{f}_i$ and $\widetilde{e}_i$ on charged symbols (or charged bipartitions in other contexts). 

\begin{remark}
The reader might be surprised by the occurrence of $d/2$ in the definition of addable/removable boxes. This comes from the fact that the charges of our symbols will come from parameters $(q^{s_1},-q^{s_2})$ in the Hecke algebra. Using the fact that $-1 = q^{d/2}$, one should rather work with symbols of  charge $(s_1,s_2+d/2)$, but that will remove the symmetry from the combinatorics of $d/2$-co-hooks and $d/2$-co-cores. We have chosen to work with the original symbol parametrizing a characteristic $0$ unipotent character, while shifting the notion of $i$-boxes to account for its behavior in quantum characteristic $d$. This way we are consistent with the usual description of unipotent $\ell$-blocks. The discrepancy also appears in Section~\ref{subsec:cataction} where the charge on bipartitions and symbols differs by $(0,d/2)$.
\end{remark}

Adding and removing boxes does not change the charge. Removing all the possible boxes yields a symbol containing in each row only consecutive integers. Such a symbol corresponds to the empty bipartition, and the charge equals $(\mathrm{max}(X_1), \mathrm{max}(X_2))$ where $(X_1,X_2)$ is the $\beta$-set of the empty bipartition.

\begin{example}
Let us consider the symbol drawn in Example~\ref{ex:symbol}. Assume $d=8$. Then there are two addable $1$-boxes, the one in the top row being a good addable $1$-box. 
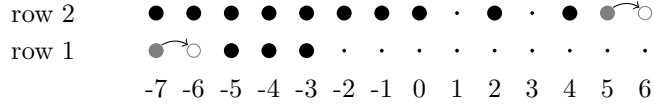
\[
\TikZ{[scale=.5]
\draw
(-7,2)node[]{\hbox{row 2}}
(-7,1)node[]{\hbox{row 1}}
(9,1)node[fill,circle,inner sep=.5pt]{}
(8,2)node[fill=gray,circle,inner sep=2pt]{}
(8,1)node[fill,circle,inner sep=.5pt]{}
(7,2)node[fill,circle,inner sep=2pt]{}
(7,1)node[fill,circle,inner sep=.5pt]{}
(6,2)node[fill,circle,inner sep=.5pt]{}
(6,1)node[fill,circle,inner sep=.5pt]{}
(5,2)node[fill,circle,inner sep=2pt]{}
(5,1)node[fill,circle,inner sep=.5pt]{}
(4,2)node[fill,circle,inner sep=.5pt]{}
(4,1)node[fill,circle,inner sep=.5pt]{}
(3,2)node[fill,circle,inner sep=2pt]{}
(3,1)node[fill,circle,inner sep=.5pt]{}
(2,2)node[fill,circle,inner sep=2pt]{}
(2,1)node[fill,circle,inner sep=.5pt]{}
(1,2)node[fill,circle,inner sep=2pt]{}
(1,1)node[fill,circle,inner sep=.5pt]{}
(0,2)node[fill,circle,inner sep=2pt]{}
(0,1)node[fill,circle,inner sep=2pt]{}
(-1,2)node[fill,circle,inner sep=2pt]{}
(-1,1)node[fill,circle,inner sep=2pt]{}
(-2,2)node[fill,circle,inner sep=2pt]{}
(-2,1)node[fill,circle,inner sep=2pt]{} 
(-3,2)node[fill,circle,inner sep=2pt]{}
%(-3,1)node[fill,circle,inner sep=.5pt]{}
(-4,2)node[fill,circle,inner sep=2pt]{}
(-4,1)node[fill=gray,circle,inner sep=2pt]{}
(9,0)node[]{6}
(8,0)node[]{5} % right region end
(7,0)node[]{4}
(6,0)node[]{3}
(5,0)node[]{2}
(4,0)node[]{1}
(3,0)node[]{0}%right region start
(2,0)node[]{-1} %left region end
(1,0)node[]{-2}
(0,0)node[]{-3}
(-1,0)node[]{-4}
(-2,0)node[]{-5}
(-3,0)node[]{-6} %left region start
(-4,0)node[]{-7};
\draw[fill=none, color=gray]
(9,2)circle(5pt)
(-3,1)circle(5pt);
 \draw [->] (8.15,2.15) to[bend left] (8.85,2.15);
  \draw [->] (-3.85,1.15) to[bend left] (-3.15,1.15);
}
\]
\end{example}

\subsubsection{Co-hooks and co-cores}
Let $e \geq 1$ be a positive integer. Given a charged symbol $\Theta = (X_1,X_2)$, a $e$-co-hook in row $k$ is a pair $(x,x-e)$ where $x \in X_k$ and $x-e \notin X_{k+1}$ (here $k+1$ must be understood modulo $2$). Removing the $e$-co-hook to $\Theta$ amounts to removing $x$ from $X_k$ and adding $x-e$ to  $X_{k+1}$ and swapping $X_1$ and $X_2$. If $\Theta$ has charge $(s_1,s_2)$ then removing a $e$-co-hook yields a symbol of charge $(s_2 \pm1, s_1 \mp1)$. A charged symbol with no $e$-co-hook is called a $e$-co-core.

\begin{example}
Let us again consider the symbol drawn in Example~\ref{ex:symbol}. There are four $4$-co-hooks, all being in the top row.
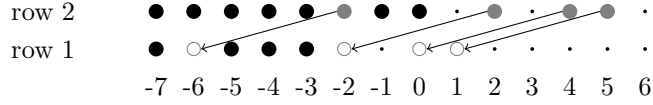
\[
\TikZ{[scale=.5]
\draw
(-7,2)node[]{\hbox{row 2}}
(-7,1)node[]{\hbox{row 1}}
(9,2)node[fill,circle,inner sep=.5pt]{}
(9,1)node[fill,circle,inner sep=.5pt]{}
(8,2)node[fill=gray,circle,inner sep=2pt]{}
(8,1)node[fill,circle,inner sep=.5pt]{}
(7,2)node[fill=gray,circle,inner sep=2pt]{}
(7,1)node[fill,circle,inner sep=.5pt]{}
(6,2)node[fill,circle,inner sep=.5pt]{}
(6,1)node[fill,circle,inner sep=.5pt]{}
(5,2)node[fill=gray,circle,inner sep=2pt]{}
(5,1)node[fill,circle,inner sep=.5pt]{}
(4,2)node[fill,circle,inner sep=.5pt]{}
(3,2)node[fill,circle,inner sep=2pt]{}
(2,2)node[fill,circle,inner sep=2pt]{}
(2,1)node[fill,circle,inner sep=.5pt]{}
(1,2)node[fill=gray,circle,inner sep=2pt]{}
(0,2)node[fill,circle,inner sep=2pt]{}
(0,1)node[fill,circle,inner sep=2pt]{}
(-1,2)node[fill,circle,inner sep=2pt]{}
(-1,1)node[fill,circle,inner sep=2pt]{}
(-2,2)node[fill,circle,inner sep=2pt]{}
(-2,1)node[fill,circle,inner sep=2pt]{} 
(-3,2)node[fill,circle,inner sep=2pt]{}
(-4,2)node[fill,circle,inner sep=2pt]{}
(-4,1)node[fill,circle,inner sep=2pt]{}
(9,0)node[]{6}
(8,0)node[]{5} 
(7,0)node[]{4}
(6,0)node[]{3}
(5,0)node[]{2}
(4,0)node[]{1}
(3,0)node[]{0}
(2,0)node[]{-1} 
(1,0)node[]{-2}
(0,0)node[]{-3}
(-1,0)node[]{-4}
(-2,0)node[]{-5}
(-3,0)node[]{-6} 
(-4,0)node[]{-7};
\draw[fill=none, color=gray]
(4,1)circle(5pt)
(3,1)circle(5pt)
(1,1)circle(5pt)
(-3,1)circle(5pt);
 \draw [->] (7.8,2) to (4.2,1);
 \draw [->] (6.8,2) to (3.2,1);
 \draw [->] (4.8,2) to (1.2,1);
 \draw [->] (0.8,2) to (-2.8,1);
}
\]
\end{example}

\subsection{Harish-Chandra series of unipotent characters}
We follow \cite{FoSr89} for the groups of type $B/C$ and \cite{W04} for the groups of type $D$.
We start by describing the cuspidal irreducible unipotent characters of $G_n(q)$.
\begin{itemize}
\item $\mathrm{SO}_{2n+1}(q)$ and $\mathrm{Sp}_{2n}(q)$ have a cuspidal unipotent 
character if and only if $n=t^2+t$ for some $t \in \mathbb{N}$. In that case it is unique, and we denote it by $\Delta_t$.
\item $\mathrm{O}_{2n}^{+}(q)$ has a cuspidal unipotent 
character if and only if $n=t^2$ for some $t \in 2\mathbb{N}$. If $t=0$, there is a unique one, denoted by $\Delta_0$. If $t \neq 0$, there are exactly two, which we denote by $\Delta_t$ and $\Delta_{-t}$. 
\item $\mathrm{O}_{2n}^{-}(q)$ has a cuspidal unipotent 
character if and only if $n=t^2$ for some $t \in 2\mathbb{N}+1$. In that case there are exactly two, which we denote
again by $\Delta_t$ and $\Delta_{-t}$. 
\end{itemize}
Note that $t^2+t$ is unchanged by the transformation $t \mapsto -t-1$, so writing $\Delta_t$ with $t \in \mathbb{Z}$ makes sense in all cases, with the convention that
$\Delta_{t} = \Delta_{-1-t}$ for groups of type $B/C$.

\smallskip

Let $n,m,t \geq 0$ be such that $G_n(q)$ has a cuspidal unipotent character $\Delta_t$. Then the unipotent characters of $G_{n+m}(q)$ above $\Delta_t$ are parametrized by bipartitions of $m$. 

\subsection{Classification of unipotent characters}\label{ssec:classification}
In order to have a global treatment of all the Harish-Chandra series, we define, for  
each $t \in \mathbb{Z}$, the following charge
$$\bsigma_t = \left\{ \begin{array}{ll}
 (t,-1-t) & \text{if $t$ is even and $G_n$ is of type $B/C$}, \\
 (-1-t,t) & \text{if $t$ is odd and $G_n$ is of type $B/C$}, \\
 (t,-t) & \text{if $t$ is even and $G_n$ is of type $D$}, \\
 (-t,t) & \text{if $t$ is odd and $G_n$ is of type ${}^2D$}. \\ \end{array}\right.$$
Then unipotent characters are classified by charged symbols of charged $\bsigma_t$, see Section~\ref{subsec:comb} for the definition or properties of symbols. In a series above $\Delta_t$, the symbols will have defect $2t+1$, $-2t-1$, $2t$, $-2t$ depending on the type of groups and the parity of $t$. 

\smallskip

Given a charged symbol $\Theta$ with charge $\bsigma_t$ ($t\in \mathbb{Z}$), we will denote by $\Delta_\Theta$ the corresponding unipotent character. Note that there are no unique such parametrization, but the one we will choose comes from a categorical action and it will satisfy the properties given in the following sections. 

\smallskip

\subsection{Unipotent blocks} Recall that the multiplicative order of $q$ in $\mathbb{k}^\times$, denoted by $d$, is assumed to be even.
A \emph{unipotent $\ell$-block} is an $\ell$-block containing at least one unipotent character. 
The partition of unipotent characters into $\ell$-blocks can be read off from the labelling of unipotent characters by symbols. By \cite{FoSr89,Sr08}, one can chose the parametrization such that two unipotent characters are in the same $\ell$-block if and only if the corresponding symbols have the same $d/2$-cocore. 

\subsection{Classification of unipotent simple modules over $\mathbb{k}$}
Under the assumptions on $q$ and $\ell$, the decomposition matrix of the unipotent blocks of classical groups have a unitriangular shape. This is proven for finite classical groups coming from connected reductive algebraic groups in \cite{BDT19}, and the case of $\mathrm{O}_{2n}^\pm(q)$ follows for example from \cite[Thm. 3.1]{FeSp23}. In particular, the parametrization of unipotent characters by charged symbols yields a parametrization of the irreducible unipotent representations over $\mathbb{k}$ as well. Given a charged symbol $\Theta$ with charge $\bsigma_t$ ($t\in \mathbb{Z}$), we will denote by $S_\Theta$ the corresponding simple representation.

\subsection{Categorical action}\label{subsec:cataction}
Let $G_n(q)$ be a finite classical group. We denote by $\mathbb{O}G_n(q)\mathsf{-umod}$ the category of finitely generated unipotent representations of $G_n(q)$ over $\mathbb{O}$. It is the direct summand of $\mathbb{O}G_n(q)\mathsf{-mod}$ corresponding to the sum of all unipotent blocks. The type of finite classical group being fixed, we denote by $\mathcal{V}$ the category
$$ \mathcal{V} := \bigoplus_{n \geq 0} \mathbb{O}G_n(q)\mathsf{-umod}.$$

Recall that we work in the \emph{unitary prime} case, where $d$, the multiplicative order of $q$ in $\mathbb{k}^\times$, is even. In that case, there is, for every $i \in \mathbb{Z}/d$, a categorical datum on $\mathcal{V}$ inducing an $\mathfrak{sl}_2$-categorical action on $\mathbb{k} \mathcal{V}$. We will denote by $E_i$ and $F_i$ the corresponding functors (over any ring of coefficient between $\mathbb{K}$,  $\mathbb{O}$ and $\mathbb{k}$). These functors are defined using Harish-Chandra induction and restriction functors, see \cite{DVV19,DVV17} for more details for their construction.

\smallskip

These various $\mathfrak{sl}_2$-categorical actions come from an $\widehat{\mathfrak{sl}}_d$-categorical action. We shall not use that fact since we will be working 
with each action separately. However, it is important to know that one can chose the parametrisation of unipotent characters in such a way that one can 
actually compute the action of each $E_i$, $F_i$ on irreducible unipotent characters and unipotent Brauer characters. More precisely, if we define
\begin{equation}
\label{eq:charge-s}
\bs_t = \bsigma_t + (0,d/2) 
\end{equation}
then there exists an isomorphism of $\widehat{\mathfrak{sl}}_d$-modules
\begin{equation}
\label{eq:fock}
 \begin{array}{rcl} 
\mathbb{C} \otimes_\mathbb{Z} K_0(\bbk \mathcal{V})   &   \mathop{\longrightarrow}\limits^\sim
 & \displaystyle\bigoplus_{t }\mathsf{F}(\bs_t)  \\[6pt] \big[\Delta_{\Theta(\blambda,\bsigma_t)}\big]  & \longmapsto & |\blambda,\bs_t\rangle
 \end{array}
 \end{equation}
where $\mathsf{F}(\bs_t)$ is the level 2 Fock space of charge $\bs_t$, and  $\big[\Delta\big] $ denotes the class of any $\ell$-reduction of the character $\Delta $. Note that depending on the type of the classical group $G_n$, the integer $t$ will run over the non-negative integers only (for type $B/C$), over all the even integers (for type $D$) or over all the odd integers (for type ${}^2 D$).

\smallskip

Note that the action of $\widehat{\mathfrak{sl}}_d$ on the Fock space $\mathsf{F}(\bs_t)$ depends only on the charge up to shifts by $(0,d)$ and $(d,0)$.

However, $\bs_t\in\bbZ^2$ is chosen so that the map
$$S_{\Theta(\blambda,\bsigma_t)} \longmapsto |\blambda,\bs_t\rangle  $$
induces an isomorphism between the crystals (which are sensitive to the charge, not only the residue class of the components of the charge). This is proven for type $B/C$ in \cite[Thm. 1.7]{DN}, and the same proof should work for type $D$ and ${}^2D$ once we have a good parametrization.

\subsection{Explicit formulas}
The categorification results contained in the previous section were proven to get an explicit description of the Harish-Chandra induction and restriction functors on the unipotent representations in both characteristic zero and $\ell$. Given a charged symbol $\Theta$ with charge $\bsigma_t$, we have 
\begin{equation}
\label{eq:action-on-delta} 
E_i\big(\Delta_\Theta \big) = \bigoplus_{\begin{subarray}{c} j \equiv i \text{ mod } d \\ \Theta \smallsetminus \Psi=j \end{subarray}} \Delta_\Psi
\end{equation}
so that $E_i$ removes all the possible removable $i$-boxes from the charged symbol $\Theta$.
\begin{equation}
\label{eq:action-on-simple} 
\mathrm{soc}\, E_i\big(S_{\Theta} \big) = \mathrm{hd}\, E_i\big(S_{\Theta}\big) = \left\{ \begin{array}{ll} 0 & \text{$\Theta$ has no good removable $i$-box},\\
S_{\Psi } & \text{if $\Psi$ is obtained from $\Theta$ by removing the good $i$-box}. \\ \end{array} \right.
\end{equation}
Similar formulas are obtained for the action of $F_i$ by adjunction.

%%%%%%%%%%%%%%%%%%%%%%%%%%%%%%%%%%%%%%%%%%%%%%%%%

\section{The diagrammatics of $d$-small symbols}\label{sec:diagrammatics} 

In this section, we adapt the combinatorics in \cite{Bru,BS1,BS3} to an appropriate class of symbols labeling unipotent characters of types $B$, $C$, $D$ and ${}^2D$. This will position us to prove the formula for decomposition numbers in Theorem \ref{thm:main} using Proposition \ref{prop:decnumbers}.

\subsection{$d$-small symbols}\label{ssec:dsmall}

Given a symbol $\Theta$, we work with its graphical representation as two rows of beads and spaces, and so we will often refer colloquially to the elements $\beta\in\Theta$ as \emph{beads} and to the elements of $(\bbZ,\bbZ)\setminus\Theta$ as \emph{spaces}.

\begin{definition}
Let $\Theta=(X_1,X_2)$ be a symbol.
For each $i=1,2$, define the interval $J_i=[c_i,d_i]\subset \bbZ$ by $$d_i=\max\{\beta\in\bbZ\;\mid\; \beta\in X_i \}$$ and $$c_i=\max\{z\in\bbZ\;\mid\; \beta\in X_i \hbox{ for all }\beta<z\}.$$
Fix $d\in2\bbN$. We say $\Theta$ is {\em $d$-small} if there exist intervals $I_1:=[a_1,b_1],I_2:=[a_2,b_2]\subset\mathbb{Z}$ satisfying the following conditions:
\begin{itemize}
\item  $b_i-a_i=\frac{d}{2}-1$ for $i=1,2$,
\item $J_i\subseteq I_i$ for $i=1,2$,
\item $b_2\equiv b_1+\frac{d}{2} \mod d$.
\end{itemize}
\end{definition}

Given a $d$-small symbol $\Theta$, we define the {\em right region} to be the region of $\Theta$ containing beads in the larger interval of integers $I_j$ , and the {\em left region} to be the region of $\Theta$ containing beads in the smaller subset of integers $I_{j+1}$ (taking the subscript mod $2$). We define the {\em middle region} to region of the symbol between the left and right regions. We note that the length of the middle region is always a multiple of $d$. The middle region contains only beads in the row containing the beads of the right region, and only spaces in the row containing the spaces of the left region. The left region always contains only beads in one row (the same row as the beads of the middle region) while the right region always contains only spaces in the opposite row. By abuse of language, we will often refer to the intervals $I_j$ and $I_{j+1}$ of $\bbZ$ as the right and left regions, and the interval of integers between them as the middle region.

\begin{example}\label{dsmall exl} 
The following symbol is $d$-small for $d=12$:
\[
\TikZ{[scale=.5]
\draw
(-7,2)node[]{\hbox{row 2}}
(-7,1)node[]{\hbox{row 1}}
(9,2)node[fill,circle,inner sep=.5pt]{}
(9,1)node[fill,circle,inner sep=.5pt]{}
(8,2)node[fill,circle,inner sep=2pt]{}
(8,1)node[fill,circle,inner sep=.5pt]{}
(7,2)node[fill,circle,inner sep=2pt]{}
(7,1)node[fill,circle,inner sep=.5pt]{}
(6,2)node[fill,circle,inner sep=.5pt]{}
(6,1)node[fill,circle,inner sep=.5pt]{}
(5,2)node[fill,circle,inner sep=2pt]{}
(5,1)node[fill,circle,inner sep=.5pt]{}
(4,2)node[fill,circle,inner sep=.5pt]{}
(4,1)node[fill,circle,inner sep=.5pt]{}
(3,2)node[fill,circle,inner sep=2pt]{}
(3,1)node[fill,circle,inner sep=.5pt]{}
(2,2)node[fill,circle,inner sep=2pt]{}
(2,1)node[fill,circle,inner sep=.5pt]{}
(1,2)node[fill,circle,inner sep=2pt]{}
(1,1)node[fill,circle,inner sep=.5pt]{}
(0,2)node[fill,circle,inner sep=2pt]{}
(0,1)node[fill,circle,inner sep=2pt]{}
(-1,2)node[fill,circle,inner sep=2pt]{}
(-1,1)node[fill,circle,inner sep=2pt]{}
(-2,2)node[fill,circle,inner sep=2pt]{}
(-2,1)node[fill,circle,inner sep=2pt]{}
(-3,2)node[fill,circle,inner sep=2pt]{}
(-3,1)node[fill,circle,inner sep=.5pt]{}
(-4,2)node[fill,circle,inner sep=2pt]{}
(-4,1)node[fill,circle,inner sep=2pt]{}
(9,0)node[]{6}
(8,0)node[]{5} 
(7,0)node[]{4}
(6,0)node[]{3}
(5,0)node[]{2}
(4,0)node[]{1}
(3,0)node[]{0}
(2,0)node[]{-1} 
(1,0)node[]{-2}
(0,0)node[]{-3}
(-1,0)node[]{-4}
(-2,0)node[]{-5}
(-3,0)node[]{-6} 
(-4,0)node[]{-7}
(8.5,0.5)node{} to (8.5,2.5)node{}
(2.5,0.5)node{} to (2.5,2.5)node{}
(-3.5,0.5)node{} to (-3.5,2.5)node{}
;
}
\]
The left region must be chosen to be $[-6,-1]$ and the right region to be $[0,5]$, and we have marked these regions with the vertical lines in the abacus diagram. The middle region is empty in this example.
\end{example}

\begin{example}\label{dsmall exl2}
We take $d=10$ and consider 
the following symbol. We see that it is $d$-small, with right region $[5,9]$ and left region $[-10,-6]$. The middle region has length $d$ in this case.
\[
\TikZ{[scale=.5]
\draw
(-14,2)node[]{\hbox{row 2}}
(-14,1)node[]{\hbox{row 1}}
(10,2)node[fill,circle,inner sep=.5pt]{}
(10,1)node[fill,circle,inner sep=.5pt]{}
(9,2)node[fill,circle,inner sep=.5pt]{}
(9,1)node[fill,circle,inner sep=.5pt]{}
(8,2)node[fill,circle,inner sep=2pt]{}
(8,1)node[fill,circle,inner sep=.5pt]{}
(7,2)node[fill,circle,inner sep=2pt]{}
(7,1)node[fill,circle,inner sep=.5pt]{}
(6,2)node[fill,circle,inner sep=.5pt]{}
(6,1)node[fill,circle,inner sep=.5pt]{}
(5,2)node[fill,circle,inner sep=2pt]{}
(5,1)node[fill,circle,inner sep=.5pt]{}
(4,2)node[fill,circle,inner sep=2pt]{}
(4,1)node[fill,circle,inner sep=.5pt]{}
(3,2)node[fill,circle,inner sep=2pt]{}
(3,1)node[fill,circle,inner sep=.5pt]{}
(2,2)node[fill,circle,inner sep=2pt]{}
(2,1)node[fill,circle,inner sep=.5pt]{}
(1,2)node[fill,circle,inner sep=2pt]{}
(1,1)node[fill,circle,inner sep=.5pt]{}
(0,2)node[fill,circle,inner sep=2pt]{}
(0,1)node[fill,circle,inner sep=.5pt]{}
(-1,2)node[fill,circle,inner sep=2pt]{}
(-1,1)node[fill,circle,inner sep=.5pt]{}
(-2,2)node[fill,circle,inner sep=2pt]{}
(-2,1)node[fill,circle,inner sep=.5pt]{}
(-3,2)node[fill,circle,inner sep=2pt]{}
(-3,1)node[fill,circle,inner sep=.5pt]{} 
(-4,2)node[fill,circle,inner sep=2pt]{}
(-4,1)node[fill,circle,inner sep=.5pt]{}
(-5,2)node[fill,circle,inner sep=2pt]{}
(-5,1)node[fill,circle,inner sep=.5pt]{}
(-6,2)node[fill,circle,inner sep=2pt]{}
(-6,1)node[fill,circle,inner sep=2pt]{}
(-7,2)node[fill,circle,inner sep=2pt]{}
(-7,1)node[fill,circle,inner sep=.5pt]{}
(-8,2)node[fill,circle,inner sep=2pt]{}
(-8,1)node[fill,circle,inner sep=2pt]{}
(-9,2)node[fill,circle,inner sep=2pt]{}
(-9,1)node[fill,circle,inner sep=2pt]{}
(-10,2)node[fill,circle,inner sep=2pt]{}
(-10,1)node[fill,circle,inner sep=.5pt]{}
(-11,2)node[fill,circle,inner sep=2pt]{}
(-11,1)node[fill,circle,inner sep=2pt]{}

(9,0)node[]{9} 
(4,0)node[]{4}
(-1,0)node[]{-1}
(-6,0)node[]{-6} 
(-11,0)node[]{-11} 
(9.5,0.5)node{} to (9.5,2.5)node{}
(4.5,0.5)node{} to (4.5,2.5)node{}
(-5.5,0.5)node{} to (-5.5,2.5)node{}
(-10.5,0.5)node{} to (-10.5,2.5)node{}
;
}
\]

\end{example}

\begin{example}\label{exl:regionchoice} Consider 
the following symbol. 
\[
\TikZ{[scale=.5]
\draw
(-13,2)node[]{\hbox{row 2}}
(-13,1)node[]{\hbox{row 1}}
(10,2)node[fill,circle,inner sep=.5pt]{}
(10,1)node[fill,circle,inner sep=.5pt]{}
(9,2)node[fill,circle,inner sep=.5pt]{}
(9,1)node[fill,circle,inner sep=.5pt]{}
(8,2)node[fill,circle,inner sep=.5pt]{}
(8,1)node[fill,circle,inner sep=.5pt]{}
(7,2)node[fill,circle,inner sep=.5pt]{}
(7,1)node[fill,circle,inner sep=.5pt]{}
(6,2)node[fill,circle,inner sep=.5pt]{}
(6,1)node[fill,circle,inner sep=.5pt]{}
(5,2)node[fill,circle,inner sep=.5pt]{}
(5,1)node[fill,circle,inner sep=2pt]{}
(4,2)node[fill,circle,inner sep=.5pt]{}
(4,1)node[fill,circle,inner sep=.5pt]{}
(3,2)node[fill,circle,inner sep=.5pt]{}
(3,1)node[fill,circle,inner sep=2pt]{}
(2,2)node[fill,circle,inner sep=.5pt]{}
(2,1)node[fill,circle,inner sep=2pt]{}
(1,2)node[fill,circle,inner sep=.5pt]{}
(1,1)node[fill,circle,inner sep=2pt]{}
(0,2)node[fill,circle,inner sep=.5pt]{}
(0,1)node[fill,circle,inner sep=2pt]{}
(-1,2)node[fill,circle,inner sep=.5pt]{}
(-1,1)node[fill,circle,inner sep=2pt]{}
(-2,2)node[fill,circle,inner sep=2pt]{}
(-2,1)node[fill,circle,inner sep=2pt]{}
(-3,2)node[fill,circle,inner sep=.5pt]{}
(-3,1)node[fill,circle,inner sep=2pt]{} 
(-4,2)node[fill,circle,inner sep=2pt]{}
(-4,1)node[fill,circle,inner sep=2pt]{}
(-5,2)node[fill,circle,inner sep=.5pt]{}
(-5,1)node[fill,circle,inner sep=2pt]{}
(-6,2)node[fill,circle,inner sep=.5pt]{}
(-6,1)node[fill,circle,inner sep=2pt]{}
(-7,2)node[fill,circle,inner sep=2pt]{}
(-7,1)node[fill,circle,inner sep=2pt]{}
(-8,2)node[fill,circle,inner sep=2pt]{}
(-8,1)node[fill,circle,inner sep=2pt]{}
(-9,2)node[fill,circle,inner sep=2pt]{}
(-9,1)node[fill,circle,inner sep=2pt]{}
(-10,2)node[fill,circle,inner sep=2pt]{}
(-10,1)node[fill,circle,inner sep=2pt]{}
(10,0)node[]{10} 
(9,0)node[]{9} 
(8,0)node[]{8} 
(7,0)node[]{7} 
(6,0)node[]{6} 
(5,0)node[]{5} 
(4,0)node[]{4}
(3,0)node[]{3}
(2,0)node[]{2}  
(1,0)node[]{1} 
(0,0)node[]{0} 
(-1,0)node[]{-1}
(-2,0)node[]{-2}
(-3,0)node[]{-3}
(-4,0)node[]{-4}   
(-5,0)node[]{-5} 
(-6,0)node[]{-6}
(-7,0)node[]{-7}
(-8,0)node[]{-8}
(-9,0)node[]{-9}
(-10,0)node[]{-10}     
;
}
\]
Take $d=16$. Then this symbol is $d$-small with four different possible choices for the left and right regions. 
They are: 
\begin{enumerate}
\item $\hbox{left region}=[-9,-2]$, $\hbox{right region}=[-1,6]$,
\item $\hbox{left region}=[-8,-1]$, $\hbox{right region}=[0,7]$,
\item $\hbox{left region}=[-7,0]$, $\hbox{right region}=[1,8]$,
\item $\hbox{left region}=[-6,1]$, $\hbox{right region}=[2,9]$.
\end{enumerate}
\end{example}

\begin{remark}\label{rmk:dsmalloncontent}
Recall that if $\Theta = \Theta(\blambda,\bsigma)$ is the symbol attached to the charged bipartition $|\blambda,\bsigma\rangle$, then $|\blambda,(\sigma_1,\sigma_2+\frac{d}{2})\rangle=:|\blambda,\bs\rangle$ is the charged bipartition for the Fock space as in (\ref{eq:fock}). Then $\Theta$ is $d$-small if and only if the set of residues of the charged contents of boxes in $|\blambda,\bs\rangle$ is  contained in a closed interval of length $\frac{d}{2}-1$ modulo $d$.
\end{remark}

\subsection{Diagrams}

\begin{definition}\label{symb2BS} Fix $d\in 2\bbN$. 
Let $\Theta$ be a $d$-small symbol. Write the right region as \linebreak $I_j=m+[1,\frac{d}{2}]\subset \bbZ$ for some $m\in\bbZ$. Let $kd$ be the length of the middle region, $k\in\bbN$. To $\Theta$ we associate an {\em up-down diagram} $\updown(\Theta)=w_1w_2\cdots w_{\frac{d}{2}}$ 
consisting of a word of length $\frac{d}{2}$ in the alphabet $\{\up,\down,\circ,\times\}$. Each $\beta_i=m+i\in I_j$ determines the $i$'th letter $w_i$ of $\updown(\Theta)$ according to the rule:
\begin{itemize}
\item if $\beta_i\in \Theta$ and $ \beta_i-kd-\frac{d}{2}$ occurs in both rows of $\Theta$, then $w_i=\times$,
\item  if $\beta_i\in \Theta$ and $ \beta_i-kd-\frac{d}{2}$ occurs only once in $\Theta$, then $w_i=\up$,
\item  if $\beta_i\notin \Theta$ and $\beta_i-kd-\frac{d}{2}$ occurs in both rows of $\Theta$, then $w_i=\down$,
\item if $\beta_i\notin \Theta$ and $ \beta-kd-\frac{d}{2}$ occurs only once in $\Theta$, then $w_i=\circ$.
\end{itemize}
\end{definition}

Thus, the up-down diagram of $\Theta$ records what happens in the left and right regions of $\Theta$, in the opposite rows where a mix of beads and spaces is possible. Putting the right region on top of the left region, the rules corresponds to the following possibilities:
\begin{equation}\label{eq:updown}
\TikZ{[scale=.5]
\draw
(0,2)node[fill,circle,inner sep=2pt]{}
(0,1)node[fill,circle,inner sep=2pt]{}
(1,2)node[fill,circle,inner sep=2pt]{}
(1,1)node[fill,circle,inner sep=.5pt]{}
(2,2)node[fill,circle,inner sep=.5pt]{}
(2,1)node[fill,circle,inner sep=2pt]{}
(3,2)node[fill,circle,inner sep=.5pt]{}
(3,1)node[fill,circle,inner sep=.5pt]{}
;
\node at (0,0)[]{$\times$};
\node at (1,0)[]{$\up$};
\node at (2,0)[]{$\down$};
\node at (3,0)[]{$\circ$};
\node at (-4.2,1)[]{Left region};
\node at (-4,2)[]{Right region};
}
\end{equation}

\begin{definition}
Given a $d$-small symbol $\Theta$, we further associate to it a {\em cup diagram} $\cupdown(\Theta)$. 
It is uniquely determined from $\updown(\Theta)$ and consists of non-crossing counterclockwise arcs and rays attached to the $\down$'s and $\up$'s of $\updown(\Theta)$ by the following recursive procedure. Start with a pair $\down\cdots \up$ which are either adjacent (i.e. $w_iw_{i+1}=\down\up$ for some $i$), or have only $\times$ and $\circ$ symbols between them.  Connect the $\down$ to the $\up$ with a curved arc (a ``cup"). Considering the $\down$ and $\up$ symbols as directional, the cup is counterclockwise-oriented as its left endpoint is $\down$ and its right endpoint is $\up$. Next, continue to connect with cups any $\down\cdots\up$ pairs which are adjacent, or have only $\times$ and $\circ$ symbols and previously constructed cups between them. When no more counterclockwise cups can be constructed by this rule, attach vertical rays below the remaining $\up$ and $\down$ symbols.
\end{definition}

\begin{example}\label{dsmall exl ctd}
We continue with Example \ref{dsmall exl}. We have $\updown(\Theta)=\up\down\times\down\up\;\up$ and the associated cup diagram $\cupdown(\Theta)$ is:
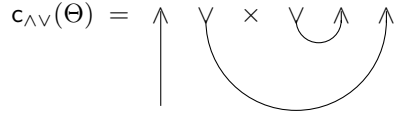
\[
\begin{tikzpicture}[scale=.6]
\node at (-2,0)[]{$\cupdown(\Theta)\;=$};
\node at (0,0)[]{$\up$};
\node at (1,0)[]{$\down$};
\node at (2,0)[]{$\times$};
\node at (3,0)[]{$\down$};
\node at (4,0)[]{$\up$};
\node at (5,0)[]{$\up$};
\draw
(3,-.1) arc(-180:0:0.5)
(1,-.1) arc(-180:0:2)
(5,-.15) node{} to (5,0.1)node{}
(4,-.15) node{} to (4,0.1)node{}
(0,.1)node{} to (0,-2)node{}
;
\end{tikzpicture}
\]
\end{example}

\begin{example}\label{exl:d28}
Let $d=28$ and consider the following $d$-small symbol:
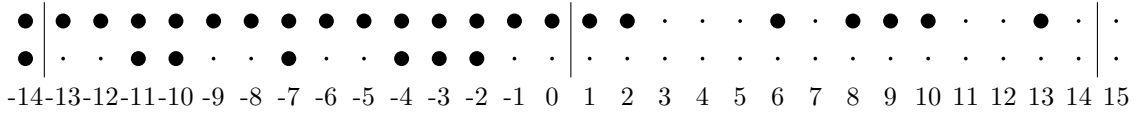
\[
\TikZ{[scale=.5]
\draw
(15,2)node[fill,circle,inner sep=.5pt]{}
(15,1)node[fill,circle,inner sep=.5pt]{}
(14,2)node[fill,circle,inner sep=.5pt]{}
(14,1)node[fill,circle,inner sep=.5pt]{}
(13,2)node[fill,circle,inner sep=2pt]{}
(13,1)node[fill,circle,inner sep=.5pt]{}
(12,2)node[fill,circle,inner sep=.5pt]{}
(12,1)node[fill,circle,inner sep=.5pt]{}
(11,2)node[fill,circle,inner sep=.5pt]{}
(11,1)node[fill,circle,inner sep=.5pt]{}
(10,2)node[fill,circle,inner sep=2pt]{}
(10,1)node[fill,circle,inner sep=.5pt]{}
(9,2)node[fill,circle,inner sep=2pt]{}
(9,1)node[fill,circle,inner sep=.5pt]{}
(8,2)node[fill,circle,inner sep=2pt]{}
(8,1)node[fill,circle,inner sep=.5pt]{}
(7,2)node[fill,circle,inner sep=.5pt]{}
(7,1)node[fill,circle,inner sep=.5pt]{}
(6,2)node[fill,circle,inner sep=2pt]{}
(6,1)node[fill,circle,inner sep=.5pt]{}
(5,2)node[fill,circle,inner sep=.5pt]{}
(5,1)node[fill,circle,inner sep=.5pt]{}
(4,2)node[fill,circle,inner sep=.5pt]{}
(4,1)node[fill,circle,inner sep=.5pt]{}
(3,2)node[fill,circle,inner sep=.5pt]{}
(3,1)node[fill,circle,inner sep=.5pt]{}
(2,2)node[fill,circle,inner sep=2pt]{}
(2,1)node[fill,circle,inner sep=.5pt]{}
(1,2)node[fill,circle,inner sep=2pt]{}
(1,1)node[fill,circle,inner sep=.5pt]{}
(0,2)node[fill,circle,inner sep=2pt]{}
(0,1)node[fill,circle,inner sep=.5pt]{}
(-1,2)node[fill,circle,inner sep=2pt]{}
(-1,1)node[fill,circle,inner sep=.5pt]{}
(-2,2)node[fill,circle,inner sep=2pt]{}
(-2,1)node[fill,circle,inner sep=2pt]{}
(-3,2)node[fill,circle,inner sep=2pt]{}
(-3,1)node[fill,circle,inner sep=2pt]{}
(-4,2)node[fill,circle,inner sep=2pt]{}
(-4,1)node[fill,circle,inner sep=2pt]{}
(-5,2)node[fill,circle,inner sep=2pt]{}
(-5,1)node[fill,circle,inner sep=.5pt]{}
(-6,2)node[fill,circle,inner sep=2pt]{}
(-6,1)node[fill,circle,inner sep=.5pt]{}
(-7,2)node[fill,circle,inner sep=2pt]{}
(-7,1)node[fill,circle,inner sep=2pt]{}
(-8,2)node[fill,circle,inner sep=2pt]{}
(-8,1)node[fill,circle,inner sep=.5pt]{}
(-9,2)node[fill,circle,inner sep=2pt]{}
(-9,1)node[fill,circle,inner sep=.5pt]{}
(-10,2)node[fill,circle,inner sep=2pt]{}
(-10,1)node[fill,circle,inner sep=2pt]{}
(-11,2)node[fill,circle,inner sep=2pt]{}
(-11,1)node[fill,circle,inner sep=2pt]{}
(-12,2)node[fill,circle,inner sep=2pt]{}
(-12,1)node[fill,circle,inner sep=.5pt]{}
(-13,2)node[fill,circle,inner sep=2pt]{}
(-13,1)node[fill,circle,inner sep=.5pt]{}
(-14,2)node[fill,circle,inner sep=2pt]{}
(-14,1)node[fill,circle,inner sep=2pt]{}
(15,0)node[]{15}
(14,0)node[]{14}
(13,0)node[]{13}
(12,0)node[]{12}
(11,0)node[]{11}
(10,0)node[]{10}
(9,0)node[]{9}
(8,0)node[]{8}
(7,0)node[]{7}
(6,0)node[]{6}
(5,0)node[]{5} 
(4,0)node[]{4}
(3,0)node[]{3}
(2,0)node[]{2}
(1,0)node[]{1}
(0,0)node[]{0}
(-1,0)node[]{-1} 
(-2,0)node[]{-2}
(-3,0)node[]{-3}
(-4,0)node[]{-4}
(-5,0)node[]{-5}
(-6,0)node[]{-6} 
(-7,0)node[]{-7}
(-8,0)node[]{-8}
(-9,0)node[]{-9}
(-10,0)node[]{-10}
(-11,0)node[]{-11}
(-12,0)node[]{-12}
(-13,0)node[]{-13}
(-14,0)node[]{-14}
(14.5,0.5)node{} to (14.5,2.5)node{}
(0.5,0.5)node{} to (0.5,2.5)node{}
(-13.5,0.5)node{} to (-13.5,2.5)node{}
;
}
\]

Putting the right region on top of the left region we have
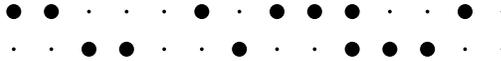
\[
\TikZ{[scale=.5]
\draw
(14,2)node[fill,circle,inner sep=.5pt]{}
(14,1)node[fill,circle,inner sep=.5pt]{}
(13,2)node[fill,circle,inner sep=2pt]{}
(13,1)node[fill,circle,inner sep=.5pt]{}
(12,2)node[fill,circle,inner sep=.5pt]{}
(12,1)node[fill,circle,inner sep=2pt]{}
(11,2)node[fill,circle,inner sep=.5pt]{}
(11,1)node[fill,circle,inner sep=2pt]{}
(10,2)node[fill,circle,inner sep=2pt]{}
(10,1)node[fill,circle,inner sep=2pt]{}
(9,2)node[fill,circle,inner sep=2pt]{}
(9,1)node[fill,circle,inner sep=.5pt]{}
(8,2)node[fill,circle,inner sep=2pt]{}
(8,1)node[fill,circle,inner sep=.5pt]{}
(7,2)node[fill,circle,inner sep=.5pt]{}
(7,1)node[fill,circle,inner sep=2pt]{}
(6,2)node[fill,circle,inner sep=2pt]{}
(6,1)node[fill,circle,inner sep=.5pt]{}
(5,2)node[fill,circle,inner sep=.5pt]{}
(5,1)node[fill,circle,inner sep=.5pt]{}
(4,2)node[fill,circle,inner sep=.5pt]{}
(4,1)node[fill,circle,inner sep=2pt]{}
(3,2)node[fill,circle,inner sep=.5pt]{}
(3,1)node[fill,circle,inner sep=2pt]{}
(2,2)node[fill,circle,inner sep=2pt]{}
(2,1)node[fill,circle,inner sep=.5pt]{}
(1,2)node[fill,circle,inner sep=2pt]{}
(1,1)node[fill,circle,inner sep=.5pt]{}
;
}
\]
hence according to \eqref{eq:updown} we get the following up-down diagram: 
$$\updown(\Theta)=\up\up\down\down\circ\up\down\up\up\times\down\down\up\;\circ$$ and 
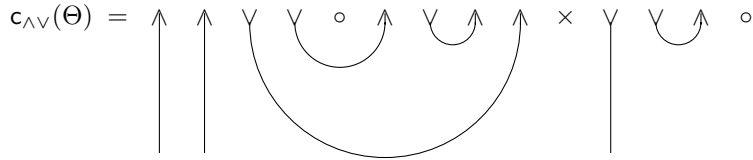
\[
\begin{tikzpicture}[scale=.6]
\node at (-1,0)[]{$\cupdown(\Theta)\;=$};
\node at (1,0)[]{$\up$};
\node at (2,0)[]{$\up$};
\node at (3,0)[]{$\down$};
\node at (4,0)[]{$\down$};
\node at (5,0)[]{$\circ$};
\node at (6,0)[]{$\up$};
\node at (7,0)[]{$\down$};
\node at (8,0)[]{$\up$};
\node at (9,0)[]{$\up$};
\node at (10,0)[]{$\times$};
\node at (11,0)[]{$\down$};
\node at (12,0)[]{$\down$};
\node at (13,0)[]{$\up$};
\node at (14,0)[]{$\circ$};
\draw
(1,0.1)node{} to (1,-3)node{}
(2,0.1)node{} to (2,-3)node{}
(4,-.1) arc(-180:0:1)
(6,-.15) node{} to (6,0.1)node{}
(7,-.1) arc(-180:0:0.5)
(8,-.15) node{} to (8,0.1)node{}
(12,-.1) arc(-180:0:0.5)
(13,-.15) node{} to (13,0.1)node{}
(3,-.1) arc(-180:0:3)
(9,-.15) node{} to (9,0.1)node{}
(11,-0.1)node{} to (11,-3)node{}
;
\end{tikzpicture}
\]

\end{example}
We remark that the rays will not play any role in computing decomposition numbers; what is important for our purposes are the cups in $\cupdown(\Theta)$. Likewise the $\times$ and $\circ$ symbols play a placeholder role. In the situation of Example \ref{exl:regionchoice} where there is more than one choice for the left and right regions, different choices are reflected in $\updown(\Theta)$ by replacing a prefix of $\times$'s with a suffix of the same number of $\circ$'s or vice versa.

\subsection{The cocore of a $d$-small symbol}

\begin{lemma}\label{lem:cocore} Let $\Theta$ be a $d$-small symbol and let $\Theta^\circ$ be its cocore. Then $\Theta^\circ$ is a $d$-small symbol, its middle region has length $0$, and $\updown(\Theta^\circ)$ is obtained from $\updown(\Theta)$ by replacing all $\up$ by $\down$.
\end{lemma}

\begin{proof}
We will induct on $k$ where $kd$ is the length of the middle region of the $d$-small symbol~$\Theta$. 
Suppose $k=0$, so the middle region has length $0$. Given an occurrence of $\up$ in $\updown(\Theta)$, there is a bead in the corresponding position of the right region and a space $\frac{d}{2}$ to its left and in the opposite row. Thus each $\up$ yields a removable cohook, and conversely, any removable cohooks are given in this way by the $\up$'s in $\updown(\Theta)$. Removing such a cohook moves the bead in question from the right region to the left region. Do this for each $\up$ in $\updown(\Theta)$, then flip the resulting symbol upside-down if the number of $\up$'s in $\updown(\Theta)$ is odd. No more cohooks can be removed from the resulting symbol, so we have found the cocore $\Theta^\circ$ of $\Theta$. It is clear that $\Theta^\circ$ has the desired description. 

\smallskip

Next, suppose by induction that if $\Psi$ is a $d$-small symbol whose middle region has length $kd$ then $\Psi^\circ$ has the desired description. Let $\Theta$ be a $d$-small symbol whose middle region has length $(k+1)d$. Partition the middle region into equal segments of length $\frac{d}{2}$. For the leftmost such segment of length $\frac{d}{2}$ (adjacent to the left region), move the $i$'th bead $\frac{d}{2}$ to the left and to the opposite row if there's a space in the $i$'th position there, thus, if $\up$ or $\circ $ is the $i$'th letter of $\updown(\Theta)$. For all other beads in the middle and right regions, they all belong to the same row. The opposite row has only empty spaces in the middle and right regions.
We move these beads $\frac{d}{2}$ to the left and to the opposite row. Finally, flip the symbol upside down if the total number of beads moved was odd. Call this symbol $\Psi$. We have removed some number of cohooks from $\Theta$ to obtain $\Psi$, so $\Psi$ has the same cocore as $\Theta$. We observe that $\Psi$ has the same pattern of beads and spaces in its left and right regions as $\Theta$. Thus $\updown(\Psi)=\updown(\Theta)$. However, the middle region of $\Psi$ has length $kd$.  Applying induction, $\Theta^\circ$ has the desired description.
\end{proof}

A $d$-small symbol $\Theta$ is determined by the length of its middle region together with $\updown(\Theta)$, up to flipping the symbol upside-down. In type $B$, that is when the defect of the symbol is odd, by convention the flip of the symbol upside-down labels the same unipotent character. Thus when $\Theta$ is the symbol of a type $B$ unipotent character, the preceding lemma completely characterizes the cocore of $\Theta$. In the case of type $D$ or $^2D$ unipotent characters, that is when the defect of the symbol is even, the lemma characterizes the cocore up to flipping it upside-down (these label two different unipotent characters in that case).

\begin{example}\label{exl:cocore} 
Take $d=12$ and the following $d$-small symbol $\Theta$:

\[
\TikZ{[scale=.4]
\draw
(-12,2)node[]{\hbox{row 2}}
(-12,1)node[]{\hbox{row 1}}
(16,2)node[fill,circle,inner sep=.5pt]{}
(16,1)node[fill,circle,inner sep=.5pt]{}
(15,2)node[fill,circle,inner sep=2pt]{} 
(15,1)node[fill,circle,inner sep=.5pt]{}
(14,2)node[fill,circle,inner sep=.5pt]{}
(14,1)node[fill,circle,inner sep=.5pt]{}
(13,2)node[fill,circle,inner sep=.5pt]{}
(13,1)node[fill,circle,inner sep=.5pt]{}
(12,2)node[fill,circle,inner sep=2pt]{}
(12,1)node[fill,circle,inner sep=.5pt]{}
(11,2)node[fill,circle,inner sep=2pt]{}
(11,1)node[fill,circle,inner sep=.5pt]{}
(10,2)node[fill,circle,inner sep=.5pt]{}
(10,1)node[fill,circle,inner sep=.5pt]{}
(9,2)node[fill,circle,inner sep=2pt]{}
(9,1)node[fill,circle,inner sep=.5pt]{}
(8,2)node[fill,circle,inner sep=2pt]{}
(8,1)node[fill,circle,inner sep=.5pt]{}
(7,2)node[fill,circle,inner sep=2pt]{}
(7,1)node[fill,circle,inner sep=.5pt]{}
(6,2)node[fill,circle,inner sep=2pt]{}
(6,1)node[fill,circle,inner sep=.5pt]{}
(5,2)node[fill,circle,inner sep=2pt]{}
(5,1)node[fill,circle,inner sep=.5pt]{}
(4,2)node[fill,circle,inner sep=2pt]{}
(4,1)node[fill,circle,inner sep=.5pt]{}
(3,2)node[fill,circle,inner sep=2pt]{}
(3,1)node[fill,circle,inner sep=.5pt]{}
(2,2)node[fill,circle,inner sep=2pt]{}
(2,1)node[fill,circle,inner sep=.5pt]{}
(1,2)node[fill,circle,inner sep=2pt]{}
(1,1)node[fill,circle,inner sep=.5pt]{}
(0,2)node[fill,circle,inner sep=2pt]{}
(0,1)node[fill,circle,inner sep=.5pt]{}
(-1,2)node[fill,circle,inner sep=2pt]{}
(-1,1)node[fill,circle,inner sep=.5pt]{}
(-2,2)node[fill,circle,inner sep=2pt]{}
(-2,1)node[fill,circle,inner sep=.5pt]{}
(-3,2)node[fill,circle,inner sep=2pt]{}
(-3,1)node[fill,circle,inner sep=.5pt]{} 
(-4,2)node[fill,circle,inner sep=2pt]{}
(-4,1)node[fill,circle,inner sep=2pt]{}
(-5,2)node[fill,circle,inner sep=2pt]{}
(-5,1)node[fill,circle,inner sep=.5pt]{}
(-6,2)node[fill,circle,inner sep=2pt]{}
(-6,1)node[fill,circle,inner sep=.5pt]{}
(-7,2)node[fill,circle,inner sep=2pt]{}
(-7,1)node[fill,circle,inner sep=2pt]{}
(-8,2)node[fill,circle,inner sep=2pt]{}
(-8,1)node[fill,circle,inner sep=.5pt]{} 
(-9,2)node[fill,circle,inner sep=2pt]{}
(-9,1)node[fill,circle,inner sep=2pt]{}
(15,0)node[]{12}
(9,0)node[]{6}
(3,0)node[]{0}
(-3,0)node[]{-6} 
(-9,0)node[]{-12}
(15.5,0.5)node{} to (15.5,2.5)node{}
(9.5,0.5)node{} to (9.5,2.5)node{}
(3.5,0.5)node{} to (3.5,2.5)node{}
(-2.5,0.5)node{} to (-2.5,2.5)node{}
(-8.5,0.5)node{} to (-8.5,2.5)node{}
;
}
\]

We have $\updown(\Theta)=\circ\times\up\circ\down\up.$ Then $\updown(\Theta^\circ)=\circ\times\down\circ\down\down$ and $\Theta^\circ$ is:

\[
\TikZ{[scale=.4]
\draw
(-9,2)node[]{\hbox{row 2}}
(-9,1)node[]{\hbox{row 1}}
(7,2)node[fill,circle,inner sep=.5pt]{}
(7,1)node[fill,circle,inner sep=.5pt]{}
(6,2)node[fill,circle,inner sep=.5pt]{}
(6,1)node[fill,circle,inner sep=.5pt]{}
(5,2)node[fill,circle,inner sep=.5pt]{}
(5,1)node[fill,circle,inner sep=.5pt]{}
(4,2)node[fill,circle,inner sep=.5pt]{}
(4,1)node[fill,circle,inner sep=.5pt]{}
(3,2)node[fill,circle,inner sep=.5pt]{}
(3,1)node[fill,circle,inner sep=.5pt]{}
(2,2)node[fill,circle,inner sep=2pt]{}
(2,1)node[fill,circle,inner sep=.5pt]{}
(1,2)node[fill,circle,inner sep=.5pt]{}
(1,1)node[fill,circle,inner sep=.5pt]{}
(0,2)node[fill,circle,inner sep=2pt]{}
(0,1)node[fill,circle,inner sep=2pt]{}
(-1,2)node[fill,circle,inner sep=2pt]{}
(-1,1)node[fill,circle,inner sep=2pt]{}
(-2,2)node[fill,circle,inner sep=2pt]{}
(-2,1)node[fill,circle,inner sep=.5pt]{}
(-3,2)node[fill,circle,inner sep=2pt]{}
(-3,1)node[fill,circle,inner sep=2pt]{}
(-4,2)node[fill,circle,inner sep=2pt]{}
(-4,1)node[fill,circle,inner sep=2pt]{}
(-5,2)node[fill,circle,inner sep=2pt]{}
(-5,1)node[fill,circle,inner sep=.5pt]{}
(-6,2)node[fill,circle,inner sep=2pt]{}
(-6,1)node[fill,circle,inner sep=2pt]{}
(6.5,0.5)node{} to (6.5,2.5)node{}
(0.5,0.5)node{} to (0.5,2.5)node{}
(-5.5,0.5)node{} to (-5.5,2.5)node{}
;
}
\]
\end{example}

\begin{lemma}\label{lem:blocks}
Let $\Theta$ be a $d$-small symbol. 
The set of symbols 
\[
\{\Theta'\mid \Theta \hbox{ and } \Theta' \hbox{ belong to the same block and have the same charge }\}
\]
consists of those $d$-small symbols with the same charge whose up-down diagrams $\updown(\Theta')$ are given by all possible permutations of the $\up$'s and $\down$'s in $\updown(\Theta)$.
\end{lemma}
\begin{proof} 
We use the description of blocks in terms of weight subspaces for the $\widehat{\mathfrak{sl}}_d$-action given by \cite[Lem. 6.10]{DVV17}. By \cite[(2.4)]{DVV17} the weight of a charged bipartition $|\blambda,\mathbf{s}\rangle$ for this action is given by 
$$ \Lambda_{s_1} + \Lambda_{s_2} - \sum_{i \in \mathbb{Z}/d} n_i(\blambda,\mathbf{s}) \alpha_i - \Delta(\mathbf{s},d) \delta$$
where $\{\Lambda_i\}_{i \in \mathbb{Z}/d}$ (resp. $\{\alpha_i\}_{i \in \mathbb{Z}/d}$) are the fundamental weights (resp. the simple roots), $\delta$ is the imaginary root, $n_i(\blambda,\mathbf{s})$ is the number of boxes of content $i$ of the charged bipartition $|\blambda,\mathbf{s}\rangle$ and $\Delta(\mathbf{s},d)$ is an integer defined in \cite[(2.3)]{DVV17}. In particular, if $|\blambda,\mathbf{s}\rangle$ and $|\bmu,\mathbf{s}\rangle$ have the same weight then $n_i(\blambda,\mathbf{s}) = n_i(\bmu,\mathbf{s})$ for all $i \in \mathbb{Z}/d$. In particular, if $\Theta(\blambda,\mathbf{s})$ is $d$-small then so is $\Theta(\bmu,\mathbf{s})$ by Remark \ref{rmk:dsmalloncontent}. 

\smallskip

Let $\Theta'$ be a symbol with the same charge as $\Theta$ and belonging to the same block. By the above argument, $\Theta'$ is $d$-small. By Lemma~\ref{lem:cocore}, the up-down diagrams $\updown(\Theta)$ and $\updown(\Theta')$ agree up to changing some $\up$'s into $\down$'s and vice versa. Since they have the same charge, the number of  $\up$'s (hence the number of $\down$'s) in their diagrams must be equal. Consequently $\updown(\Theta)$ and $\updown(\Theta')$ differ only by a permutation of $\up$'s and $\down$'s.

\smallskip
Conversely, swapping one $\up$ with one $\down$ can be obtained as a succession of removing $2k+1$ co-hooks starting from the right region followed by a succession of adding $2k+1$ co-hooks starting from the left region, where $k$ is such that $kd$ is the size of the middle region. 
\end{proof}

\begin{example}
Take $d=10$ and consider the following $d$-small symbol $\Theta$:
\[
\TikZ{[scale=.4]
\draw
(-9,2)node[]{\hbox{row 2}}
(-9,1)node[]{\hbox{row 1}}
(5,2)node[fill,circle,inner sep=.5pt]{}
(5,1)node[fill,circle,inner sep=.5pt]{}
(4,2)node[fill,circle,inner sep=2pt]{}
(4,1)node[fill,circle,inner sep=.5pt]{}
(3,2)node[fill,circle,inner sep=.5pt]{}
(3,1)node[fill,circle,inner sep=.5pt]{}
(2,2)node[fill,circle,inner sep=2pt]{}
(2,1)node[fill,circle,inner sep=.5pt]{}
(1,2)node[fill,circle,inner sep=2pt]{}
(1,1)node[fill,circle,inner sep=.5pt]{}
(0,2)node[fill,circle,inner sep=2pt]{}
(0,1)node[fill,circle,inner sep=.5pt]{}
(-1,2)node[fill,circle,inner sep=2pt]{}
(-1,1)node[fill,circle,inner sep=.5pt]{}
(-2,2)node[fill,circle,inner sep=2pt]{}
(-2,1)node[fill,circle,inner sep=2pt]{}
(-3,2)node[fill,circle,inner sep=2pt]{}
(-3,1)node[fill,circle,inner sep=.5pt]{}
(-4,2)node[fill,circle,inner sep=2pt]{}
(-4,1)node[fill,circle,inner sep=2pt]{}
(-5,2)node[fill,circle,inner sep=2pt]{}
(-5,1)node[fill,circle,inner sep=.5pt]{}
(-6,2)node[fill,circle,inner sep=2pt]{}
(-6,1)node[fill,circle,inner sep=2pt]{}
(4,0)node[]{4}
(-1,0)node[]{-1}
(-6,0)node[]{-6}
(4.5,0.5)node{} to (4.5,2.5)node{}
(-0.5,0.5)node{} to (-0.5,2.5)node{}
(-5.5,0.5)node{} to (-5.5,2.5)node{}
;
}
\]
We have $\updown(\Theta)=\up\times\up\down\up$. There are four permutations of the set of $\up$'s and $\down$'s in $\updown(\Theta)$: 
\[
\{\down\times\up\up\up,\;\; \up\times\down\up\up,\;\; \up\times\up\down\up,\;\; \up\times\up\up\down\}.
\]
By permuting the beads in the left and right regions of the symbol $\Theta$ to match these permutations of its up-down diagram, we obtain all the symbols in the same series and block. For instance, $\down\times\up\up\up$ gives the symbol
\[
\TikZ{[scale=.4]
\draw
(-9,2)node[]{\hbox{row 2}}
(-9,1)node[]{\hbox{row 1}}
(5,2)node[fill,circle,inner sep=.5pt]{}
(5,1)node[fill,circle,inner sep=.5pt]{}
(4,2)node[fill,circle,inner sep=2pt]{}
(4,1)node[fill,circle,inner sep=.5pt]{}
(3,2)node[fill,circle,inner sep=2pt]{}
(3,1)node[fill,circle,inner sep=.5pt]{}
(2,2)node[fill,circle,inner sep=2pt]{}
(2,1)node[fill,circle,inner sep=.5pt]{}
(1,2)node[fill,circle,inner sep=2pt]{}
(1,1)node[fill,circle,inner sep=.5pt]{}
(0,2)node[fill,circle,inner sep=.5pt]{}
(0,1)node[fill,circle,inner sep=.5pt]{}
(-1,2)node[fill,circle,inner sep=2pt]{}
(-1,1)node[fill,circle,inner sep=.5pt]{}
(-2,2)node[fill,circle,inner sep=2pt]{}
(-2,1)node[fill,circle,inner sep=.5pt]{}
(-3,2)node[fill,circle,inner sep=2pt]{}
(-3,1)node[fill,circle,inner sep=.5pt]{}
(-4,2)node[fill,circle,inner sep=2pt]{}
(-4,1)node[fill,circle,inner sep=2pt]{}
(-5,2)node[fill,circle,inner sep=2pt]{}
(-5,1)node[fill,circle,inner sep=2pt]{}
(-6,2)node[fill,circle,inner sep=2pt]{}
(-6,1)node[fill,circle,inner sep=2pt]{}
(4,0)node[]{4}
(-1,0)node[]{-1}
(-6,0)node[]{-6}
(4.5,0.5)node{} to (4.5,2.5)node{}
(-0.5,0.5)node{} to (-0.5,2.5)node{}
(-5.5,0.5)node{} to (-5.5,2.5)node{}
;
}
\]

\end{example}

\subsection{Action of the crystal operators on $d$-small symbols} 
In this section, we explain how the action of the crystal operators $\tilde{e}_i$ on symbols translates to $\updown(\Theta)$ for a $d$-small symbol $\Theta$. 

\smallskip

To describe the action of crystal operators using $\updown(\Theta)$, we consider $i\in\bbZ/d\bbZ$ where $i$ labels the residue mod $d$ of a position in the right region if the right region lies in row $1$ of $\Theta$. Otherwise, the left region lies in row $1$ of $\Theta$, and then we let $i$ label the residue mod $d$ of a position in the left region. Match up the letters $w_1,\ldots, w_{\frac{d}{2}}$ in $\updown(\Theta)$ from left to right with the residues of the appropriate region. In Example \ref{exl:cocore}, we'd have:
\[
\TikZ{[scale=.5]
\draw
(1,1)node[]{1}
(2,1)node[]{2}
(3,1)node[]{3}
(4,1)node[]{4}
(5,1)node[]{5}
(6,1)node[]{6}
(1,0)node[]{$\circ$}
(2,0)node[]{$\times$}
(3,0)node[]{$\up$}
(4,0)node[]{$\circ$}
(5,0)node[]{$\down$}
(6,0)node[]{$\up$}
;
}
\]
We then consider the two adjacent letters $w_j, w_{j+1}\in\{\up,\down,\circ,\times\}$ in $\updown(\Theta)$ corresponding to the residues $i,i+1$:
\begin{equation}\label{eq:annoyingconvention}
\TikZ{[scale=.5]
\draw
(0,1)node[]{$i$}
(1.5,1)node[]{$i+1$}
(0,0)node[]{$w_j$}
(1.5,0)node[]{$w_{j+1}$}
;
}
\end{equation}
Continuing our example and taking $i=5$, we'd consider:
\[
\TikZ{[scale=.5]
\draw
(0,1)node[]{5}
(1,1)node[]{6}
(0,0)node[]{$\down$}
(1,0)node[]{$\up$}
;
}
\]

Then the effect of $\tilde{e}_i$ on $\updown(\blambda)$ is given by acting locally on these two letters as described in the next lemma.

\begin{lemma}\label{lem:crysops}
Let $\Theta$ be a $d$-small symbol and let $\updown(\Theta)$ be its up-down diagram. Label the letters of $\updown(\Theta)$ with the residues of the appropriate region as described above. The action of the crystal operators $\tilde{e}_i$ and $\tilde{f}_i$ on $\Theta$ is computed locally in $\updown(\Theta)$ on $w_j w_{j+1}$ as follows:

\begin{equation} \label{eq:eionsimples}
 \xymatrix{ \star \times  \ar@/^/[r]^{\widetilde{e}_i}
 \ar@/_/@{<-}[r]_{\widetilde{f}_i}  & \times \star & \circ \star \ar@/^/[r]^{\widetilde{e}_i}
 \ar@/_/@{<-}[r]_{\widetilde{f}_i}  & \star \circ & \circ \times  \ar@/^/[r]^{\widetilde{e}_i}
 \ar@/_/@{<-}[r]_{\widetilde{f}_i}  & \ar@/^/[r]^{\widetilde{e}_i}
 \ar@/_/@{<-}[r]_{\widetilde{f}_i}  \down \up &  \times \circ }
 \end{equation}
for $\star \in \{\up,\down\}$. The action on $\up \down $ is zero. 

\end{lemma}
\begin{proof}
Suppose the Kashiwara $i$-word of $\Theta$ is non-empty for some $i$. Then it consists of at most two letters from the alphabet $\{+,-\}$, with at most one contributed by an addable or removable box in the left region and at most one contributed by an addable or removable box in the right region of $\Theta$. We need to check that if the $i$-word has two letters, then the letter contributed by the right region is larger than the letter contributed by the left region.

\smallskip

Suppose that the right region occurs in row $2$. Recall that $|\blambda,\bs_t\rangle=|\blambda,\bsigma_t+(0,\frac{d}{2})\rangle$. Then any addable or removable $i$-box in $\lambda^2$ has charged content at least $d$ greater than any addable or removable $i$-box of $\lambda^1$. Thus when comparing addable and removable $i$-boxes of $\blambda$, the one in $\lambda^2$ is larger than the one in $\lambda^1$. Next, suppose that the right region occurs in row $1$. If the middle region has length at least $d$, then as in the previous case, an addable or removable $i$-box of $\lambda^1$ will have charged content at least $d$ greater than an addable or removable $i$-box of $\lambda^2$, so will be larger. Otherwise, the right region is $\frac{d}{2}$ to the right of the left region. Then in $|\blambda,\bs_t\rangle$, the addable/removable $i$-boxes of $\lambda^1$ and $\lambda^2$ have the same charged content. In this situation, the box in $\lambda^1$ is consider larger. So again, the letter contributed to the $i$-word from the right region is larger.
\end{proof}

The lemma implies that the class in the Grothendieck group of a simple module corresponding to a $d$-small symbol lies in a representation of $\mathfrak{sl}_2$ of dimension at most $3$. The action on the standard modules is the same as for simple modules unless we are in the $3$-dimensional representation of $\mathfrak{sl}_2$, where we have
\begin{equation}\label{eq:eionstandard}
 e_i \big[\Delta_{\circ \times}\big] = \big[\Delta_{\down \up}\big] + \big[\Delta_{\up \down}\big] = f_i\big[\Delta_{\times \circ}\big]
 \end{equation}
 in the Grothendieck group. Here the subscripts correspond to the changes in the $\up \down$-diagram of the corresponding symbol.

\begin{example} Continuing with Example \ref{exl:cocore}, the lemma tells us that $$\tilde{e}_5\left(\circ\times\up\circ\down\up\right)=\circ\times\up\circ\times\circ$$ and thus $\tilde{e}_5$ acts on $\Theta$ by moving the rightmost bead in the symbol one position to the left.
\end{example}

\subsection{Highest weights for the $\widehat{\mathfrak{sl}}_d$-action}

For every charge $\bsigma=(\sigma_1,\sigma_2)\in\bbZ^2$, the symbol $|\emptyset.\emptyset,\bsigma\rangle$ labels a cuspidal unipotent character of the appropriate group (depending on the charge of the symbol) in characteristic $0$, see \S\ref{ssec:classification}. When reduced module $\ell$, it remains cuspidal \cite{DM}.

\smallskip

We now explain the classification of modular cuspidals labeled by non-empty $d$-small symbols. The bipartition of such a symbol is a rectangle concentrated in one component, with a condition on the charged content of the removable box of the rectangle. In the case that $\Theta$ is a $d$-small symbol, it follows from \cite{DVV17,Nor} and the rule for computing the Heisenberg or $\mathfrak{sl}_\infty$ crystal \cite{Ger,GerNor} that $S_\Theta$ is cuspidal if and only if $[S_\Theta]$ is a highest weight vector for $\widehat{\mathfrak{sl}}_d$, that is, if and only if $E_iS_\Theta=0$ for all $i\in\bbZ/d$. However, we will not need the condition that $S_\Theta$ is cuspidal, but only the condition of its being annihilated by every $E_i$. The description of such $\Theta$ follows immediately from Lemma \ref{lem:crysops}.

\begin{corollary}\label{cor:cusp}
Let $\Theta$ be a $d$-small symbol. We have
 $\tilde{e}_i\Theta=0$ for all $i\in\bbZ/d$ if and only if in $\updown(\blambda)$, all $\times$'s precede all $\up$'s which precede all $\down$'s which precede all $\circ$'s.
\end{corollary}

\begin{remark}
Let $\Theta$ be a $d$-small symbol on which $\tilde{e}_i$ acts by $0$ for all $i\in\bbZ/d$. Set $w=\#\{\up\hbox{'s in }\updown(\blambda)\}$ and $h=\#\{\down\hbox{'s in }\updown(\blambda)\}$.  Set $\lambda=(w^h)$.
Then $\Theta=\Theta(\blambda,\bsigma)$ where $\blambda=\lambda.\emptyset$ if the left region is in row $1$ and $\blambda=\emptyset.\lambda$ if the left region is in row $2$. The charge $\bsigma=(\sigma_1,\sigma_2)$ is given by: \begin{enumerate}
\item\label{cusp1} $\sigma_1+w-h+kd=\sigma_2+\frac{d}{2}$ for some $k\in\mathbb{N}\setminus\{0\}$ if $\blambda=(w^h).\emptyset$,
\item\label{cusp2} $\sigma_1=\sigma_2+\frac{d}{2}+w-h+kd$ for some $k\in\mathbb{N}$ if $\blambda=\emptyset.(w^h)$.
\end{enumerate}
\end{remark}

\begin{example}\label{exl:d22}
We take $d=22$ and consider the following $d$-small symbol:
\[
\TikZ{[scale=.4]
\draw
(-13,1.5)node[]{$\Theta=$}
(14,2)node[]{\hbox{row 2}}
(14,1)node[]{\hbox{row 1}}
(12,2)node[fill,circle,inner sep=.5pt]{}
(12,1)node[fill,circle,inner sep=.5pt]{}
(11,2)node[fill,circle,inner sep=.5pt]{}
(11,1)node[fill,circle,inner sep=.5pt]{}
(10,2)node[fill,circle,inner sep=.5pt]{}
(10,1)node[fill,circle,inner sep=.5pt]{}
(9,2)node[fill,circle,inner sep=.5pt]{}
(9,1)node[fill,circle,inner sep=.5pt]{}
(8,2)node[fill,circle,inner sep=.5pt]{}
(8,1)node[fill,circle,inner sep=.5pt]{}
(7,2)node[fill,circle,inner sep=.5pt]{}
(7,1)node[fill,circle,inner sep=.5pt]{}
(6,2)node[fill,circle,inner sep=.5pt]{}
(6,1)node[fill,circle,inner sep=.5pt]{}
(5,2)node[fill,circle,inner sep=.5pt]{}
(5,1)node[fill,circle,inner sep=.5pt]{}
(4,1)node[fill,circle,inner sep=2pt]{}
(4,2)node[fill,circle,inner sep=.5pt]{}
(3,1)node[fill,circle,inner sep=2pt]{}
(3,2)node[fill,circle,inner sep=.5pt]{}
(2,1)node[fill,circle,inner sep=2pt]{}
(2,2)node[fill,circle,inner sep=.5pt]{}
(1,1)node[fill,circle,inner sep=2pt]{}
(1,2)node[fill,circle,inner sep=.5pt]{}
(0,1)node[fill,circle,inner sep=2pt]{}
(0,2)node[fill,circle,inner sep=.5pt]{}
(-1,1)node[fill,circle,inner sep=2pt]{}
(-1,2)node[fill,circle,inner sep=.5pt]{}
(-2,1)node[fill,circle,inner sep=2pt]{}
(-2,2)node[fill,circle,inner sep=2pt]{}
(-3,1)node[fill,circle,inner sep=2pt]{}
(-3,2)node[fill,circle,inner sep=2pt]{} 
(-4,1)node[fill,circle,inner sep=2pt]{}
(-4,2)node[fill,circle,inner sep=2pt]{}
(-5,1)node[fill,circle,inner sep=2pt]{}
(-5,2)node[fill,circle,inner sep=2pt]{}
(-6,1)node[fill,circle,inner sep=2pt]{}
(-6,2)node[fill,circle,inner sep=2pt]{}
(-7,1)node[fill,circle,inner sep=2pt]{}
(-7,2)node[fill,circle,inner sep=.5pt]{}
(-8,1)node[fill,circle,inner sep=2pt]{}
(-8,2)node[fill,circle,inner sep=.5pt]{}
(-9,1)node[fill,circle,inner sep=2pt]{}
(-9,2)node[fill,circle,inner sep=.5pt]{}
(-10,1)node[fill,circle,inner sep=2pt]{}
(-10,2)node[fill,circle,inner sep=2pt]{}
(-11,1)node[fill,circle,inner sep=2pt]{}
(-11,2)node[fill,circle,inner sep=2pt]{}

(11,0)node[]{11} 
(0,0)node[]{0}
(-11,0)node[]{-11} 
(11.5,0.5)node{} to (11.5,2.5)node{}
(0.5,0.5)node{} to (0.5,2.5)node{}
(-10.5,0.5)node{} to (-10.5,2.5)node{}
;
}
\]
Then $\updown(\Theta)=\times \up\up\up\down\down\down\down\down\circ\circ$. By Corollary \ref{cor:cusp}, $\tilde{e}_i\Theta=0$ and therefore $E_i S_\Theta=0$ for all $i\in\bbZ$. In partition notation, $\blambda=\emptyset.(3^5)$ and the unipotent character labeled by $\Theta$ belongs to the $B_{4^2+4}$-series.
\end{example}

\section{Decomposition numbers in a $d$-small Harish-Chandra series of a block}\label{sec:proof}

Our main theorem is a closed combinatorial formula for the entries of the square submatrix of the decomposition matrix cut out by a $d$-small Harish-Chandra series within a block.

\begin{theorem}\label{thm:main}
Suppose $\Theta$ and $\Psi$ are symbols of the same charge, and that $\Theta$ is $d$-small. Then
\[
[P_\Theta : \Delta_\Psi]=\begin{cases} 1 \hbox{ if }\updown(\Psi) \hbox{ is obtained from } \updown(\Theta) \hbox{ by reversing the orientation }\\ \quad\hbox{on a subset of the cups of }\cupdown(\Theta), \\ 0 \hbox{ otherwise}. \end{cases}
\]
\end{theorem}

\begin{proof}
We can assume without loss of generality that $\Theta$ and $\Psi$ are in the same block, otherwise the decomposition number is zero. Consequently, $\Psi$ is also $d$-small by Lemma~\ref{lem:blocks}. If $S_\Theta$ is annihilated by $E_i$ for every $i\in\bbZ/d$, then $\updown(\Theta)$ has no cup by Corollary~\ref{cor:cusp} ; on the other hand, we claim that any charged symbol $\Psi$ in the same block and with the same charge as $\Theta$ will satisfy $\Psi  \nleq \Theta$ for the order defined in \cite[\S1.1.3]{DN}. Indeed, by Corollary~\ref{cor:cusp}, $\updown(\Theta)=\times\cdots\times\up\cdots\up\down\cdots\down\circ\cdots\circ$. Since $\updown(\Psi)$ is a non-trivial permutation of the $\up$'s and $\down$'s of $\updown(\Theta)$, some $\up$ in $\updown(\Psi)$ is to the right of every $\up$ in $\updown(\Theta)$. Writing $\Theta=(X_1,X_2)$ and $\Psi=(Y_1,Y_2)$, it follows that $\max(Y_1,Y_2)>\max(X_1,X_2)$ , so $\Psi\nleq \Theta$.
Therefore by unitriangularity $[P_\Theta : \Delta_\Psi] = 0$ unless $\Psi = \Theta$.

\smallskip

Now assume that $S_\Theta$ is not annihilated by every $E_i$. Then there exists $i \in \mathbb{Z}/d$ and $n \in \{1,2\}$ such that $E_i^n S_\Theta \neq 0$ and $E_i^{n+1} S_\Theta = 0$. By Proposition~\ref{prop:decnumbers} and the remark following Lemma~\ref{lem:crysops}
 we have
$$ [P_\Theta : \Delta_\Psi] = [P_{\widetilde{e}_i^n \Theta} : E_i^{(n)} \Delta_\Psi].$$
Write $\Theta' = \widetilde{e}_i^n \Theta$. By Lemma \ref{lem:crysops}, $\Theta'$ is again $d$-small. By the rules in \eqref{eq:eionsimples}, the only possible configurations in positions $i,i+1$ in $\updown(\Theta')$ are $\star \circ$ with $\star \in \{\times,\up,\down\}$ and $\times \star$ with $\star \in \{\circ,\up,\down\}$. In particular the configurations $\up \down$ and $\down \up$ cannot occur.

\smallskip

Assume by induction on the rank of the group that the theorem holds for $P_{\Theta'}$. The multiplicity $[P_\Theta : \Delta_\Psi]$ is non-zero if and only if there exists a charged symbol $\Psi'$ such that $\Delta_{\Psi'}$ occurs in $E_i^{(n)} \Delta_\Psi$ and such that $\updown(\Psi')$ is obtained from $\updown(\Theta')$ by reversing the orientation on a subset of cups of $\cupdown(\Theta')$. The constraints on $\updown(\Theta')$ given in the previous paragraph force $E_i^{(n)} \Delta_\Psi = \Delta_{\Psi'}$ by \eqref{eq:eionstandard}, so that $[P_\Theta : \Delta_\Psi] = 1$. 

\smallskip

In the configuration $\times \circ$ for $\Psi'$ (and hence $\Theta'$), suppose we are in the case $n=1$. We must have the configuration $\down \up$ for $\Theta$, as $\Theta=\tilde{f}_i\tilde{e}_i\Theta=\tilde{f}_i\Theta'$. That is, in $\updown(\Theta)$ we have $w_jw_{j+1}=\down\up$, where $w_j$ corresponds to residue $i$ as in (\ref{eq:annoyingconvention}). Hence in $\cupdown(\Theta)$ there is a cup connecting $w_j$ and $w_{j+1}$. However, in $\updown(\Theta')$ we have $w_j'w_{j+1}'=\times\circ$ so there is no cup connecting $w_j'$ and $w_{j+1}'$ in $\cupdown(\Theta')$. Since all other letters in $\updown(\Theta)$ and $\updown(\Theta')$ are the same, the cup diagram $\cupdown(\Theta)$ is thus obtained from the cup diagram $\cupdown(\Theta')$ by inserting a cup at positions $i$ and $i+1$. On the other hand, by (\ref{eq:eionstandard}) there are two possibilities for $\Psi$ in positions $i$ and $i+1$, namely $\down \up$ and $\up \down$. It follows that $\updown(\Psi)$ is obtained from $\updown(\Theta)$ by reversing the orientation on a subset of cups of $\cupdown(\Theta)$.
\smallskip

For all the other configurations, $\Psi$ is uniquely determined from $\Psi'$. We conclude that in all the cases, we have that $\updown(\Psi)$ is obtained from $\updown(\Theta)$ by reversing the orientation on a subset of cups of $\cupdown(\Theta)$.
\end{proof}

In the setting of Theorem \ref{thm:main}, let $h$ be the number of $\up$'s and $w$ be the number of $\down$'s in the up-down diagrams of $d$-small symbols in the same block. The rule for computing the decomposition numbers presented in Theorem \ref{thm:main} is the same as Brundan-Stroppel's rule computing the multiplicities of Verma modules in projective modules in the highest weight category of finite-dimensional representations of the Khovanov arc algebra $K_h^w$ \cite{BS1}. Stroppel showed that the latter module category is equivalent to the category of perverse sheaves on the Grassmannian $\mathrm{Gr}(h,h+w)$ of $h$-planes in $\mathbb{C}^{h+w}$, which in turn is equivalent to the principal block of the parabolic category $\mathcal{O}^{\mathfrak{p}}$ for the parabolic $\mathfrak{p}$ with Levi $\mathfrak{gl}_h\times\mathfrak{gl}_w$ in $\mathfrak{gl}_{h+w}$ \cite{Catharina}. This identifies the decomposition numbers in question with the value of parabolic Kazhdan-Lusztig polynomials evaluated at $1$. In \cite{BDFHN}, another way of obtaining this rule uses an oriented version of the Temperley-Lieb algebra.

\smallskip

We deduce that the decomposition numbers can be computed from Kazhdan--Lusztig polynomials. 

\begin{corollary}\label{cor:final}
The entries of the square submatrix of the decomposition matrix given by the formula of Theorem \ref{thm:main} are given by parabolic Kazhdan-Lusztig polynomials of type \linebreak $(W,P)=(S_{h+w}, S_h\times S_{w})$ evaluated at $1$.
\end{corollary}

\bibliographystyle{amsalpha}
%%%%%%%%%%%%%%%%%%%%%%%%%%%%%%%%%%%%%%%%%%%%%%%%%%%%%%%%%%%%%%%

\end{document}